\documentclass{article}

\usepackage{epsfig,xspace}
\usepackage{latexsym}
\usepackage[matrix,arrow,curve]{xy}\xyoption{all}

\usepackage[usenames,dvipsnames,svgnames,table]{xcolor}
\usepackage{amssymb,amsmath,latexsym,amsfonts,amsthm,alltt}
\usepackage[usenames,dvipsnames,svgnames,table]{xcolor}
\usepackage[dvips,usenames]{stmaryrd}
\usepackage{enumitem}
\usepackage{url}
\usepackage{graphicx}
\usepackage{float}
\usepackage{amstext}
\usepackage{cite}
\usepackage{hyperref}
\usepackage{MnSymbol}

\newtheorem{defn}{Definition}[section]
\newtheorem{rem}[defn]{Remark}
\newtheorem{thm}[defn]{Theorem}
\newtheorem{lemma}[defn]{Lemma}
\newtheorem{prop}[defn]{Proposition}
\newtheorem{coro}[defn]{Corollary}

\newcommand{\ra}{\rightarrow}
\newcommand{\lra}{\longrightarrow}

\newcommand{\Ra}{\Rightarrow}

\newcommand{\midsp}{\;|\;}
\newcommand{\sub}[2]{#1_{{}_{#2}}}
\newcommand{\telos}{\hfill$\Box$}
\newcommand{\type}[1]{{\tt #1}}

\newcommand{\iso}{\backsimeq}

\newcommand{\val}[1]{\mbox{$\lsem #1\rsem$}}
\newcommand{\forces}{\Vdash}
\newcommand{\dforces}{\forces^{\!\!\partial}}
\newcommand{\yvval}[1]{\mbox{$\llparenthesis #1 \rrparenthesis $}}

\newcommand{\infrule}[2]{\frac{\mbox{\rm $#1$}}{\mbox{\rm $#2$}}}
\newcommand{\proves}{\vdash}

\newcommand{\upv}{\upVdash}
\newcommand{\rperp}{\mbox{${}^{\upv}$}}

\newcommand{\gphi}{{\mathcal  G}(Y)}
\newcommand{\gpsi}{{\mathcal  G}(X)}

\newcommand{\lperp}{{}\rperp}

\DeclareMathOperator{\filt}{Filt}
\DeclareMathOperator{\idl}{Idl}

\newcommand{\sqleq}{\sqsubseteq}

\title{Choice-free Dualities for Lattice Expansions:\\
Application to Logics with a Negation Operator}
\author{Chrysafis (Takis) Hartonas\\
Department of Digital Systems\\
University of Thessaly, Greece\\
$\type{hartonas@uth.gr}$}
\date{}

\begin{document}
\maketitle

\begin{abstract}
Constructive dualities have been recently proposed for some lattice based algebras and a related project has been outlined by Holliday and Bezhanishvili, aiming at obtaining ``choice-free spatial dualities for other classes of algebras [$\ldots$], giving rise to choice-free completeness proofs for non-classical logics''.

We present in this article a way to complete the Holliday-Bezhanishvili project  (uniformly, for any normal lattice expansion). This is done by recasting in a choice-free manner recent relational representation and duality results by the author. These results addressed the general representation and duality problem for lattices with quasi-operators,  extending the J\'{o}nsson-Tarski approach for BAOs, and Dunn's follow up approach for distributive generalized Galois logics, to contexts where distributivity may not be assumed. To illustrate, we apply the framework  to lattices (and their logics) with some form or other of a (quasi)complementation operator, obtaining correspondence results and canonical extensions in relational frames and choice-free dualities for lattices with a minimal, or a Galois quasi-complement, or involutive lattices, including De Morgan algebras, as well as Ortholattices and Boolean algebras, as special cases.
\end{abstract}

\section{Introduction}
\label{intro section}
\subsection{Background and Motivation}
Choice-free dualities have been lately proposed for Boolean algebras by Holliday and Bezhanishvili \cite{choice-free-BA}, for Ortholattices by MacDonald and Yamamoto \cite{choice-free-Ortho}, for modal lattices by Bezhanishvili, Dmitrieva, de Groot and Moraschini \cite{choice-free-dmitrieva-bezanishvili} and for De Vries algebras by Massas \cite{choice-free-deVries}.  These are part of a project, outlined by Holliday and Bezhanishvili and aiming at obtaining ``choice-free spatial dualities for other classes of algebras [$\ldots$], giving rise to choice-free completeness proofs for non-classical logics'' \cite[page~45]{choice-free-BA}. The project has its origins in Holliday's `possibility frames' for modal logic \cite{holliday-posibility-frames}, as noted in \cite{choice-free-BA}.

Choice was also assumed in this author's \cite{sdl,duality2} (Alexander's subbasis lemma, whose proof uses Zorn's lemma, was used to prove compactness of the space), but that use of choice is inessential and, in the present article, we recast the duality in a choice-free manner, switching from a Stone to a spectral topology. In addition, the case where the underlying lattice is distributive, or a Boolean algebra, was not treated in \cite{duality2} and we fill in for this gap here.

The results by this author \cite{dloa}, and with Dunn \cite{sdl}, are related to the Moshier-Jipsen duality \cite{Moshier2014a,Moshier2014b}, some detail on these relations has been presented in \cite[Remark~4.2, Remark~4.8]{duality2} and we revisit the issue in Proposition \ref{fsat prop} in this article. To put it briefly, both this author's recent duality results and \cite{Moshier2014a,Moshier2014b} have \cite{sdl,dloa} as their shared background, on which they build, with Moshier and Jipsen stating as their explicit objective to provide purely topological representation and duality results. They achieve this objective, freeing their dual spaces of any additional, non-topological structure. This is a duality approach in the spirit of Stone's original result \cite{stone2} on distributive lattice representation, as opposed to Priestley's representation, where an additional partial order structure is assumed in the dual space and to which the duality we presented in \cite{duality2} is closer in spirit.

The duality of \cite{Moshier2014a,Moshier2014b} is not intended, nor does it lend itself for logic related, relational semantics  applications, which are at the focus of this author's research. The main reasons are, first, that admissible interpretations take their values in stable (called saturated in \cite{Moshier2014a,Moshier2014b}) sets and their topological specification requires second-order principles as the closure of a set is defined as the intersection of all open (in the topology) filters of the space (which is partially ordered by the specialization order induced by the topology). Second, whereas  for semantic purposes one typically aims for first-order definable classes of relational frames, the choice made in \cite{Moshier2014a,Moshier2014b} is to represent lattice quasi-operators as strongly continuous and meet preserving point operators on the dual topological space. In addition, and as far as the purposes of the current article are concerned, it should be mentioned that according to \cite{choice-free-dmitrieva-bezanishvili} the Moshier-Jipsen duality is not choice-free and it has been Dmitrieva's contribution (see \url{ https://eprints.illc.uva.nl/id/eprint/1811/1/MoL-2021-18.text.pdf}) to provide a choice-free version of the duality of \cite{Moshier2014a}.

In recent years, this author has pursued a project of extending J\'{o}nsson and Tarski's framework for Boolean algebras with operators (BAOs) \cite{jt1,jt2} and Dunn's follow up research on distributive lattices with quasi-operators (the algebraic models of Dunn's generalized Galois logics, or gaggles) \cite{dunn-ggl,dunn-partial,gglbook} to the case of general lattices with quasi-operators (normal lattice operators, in our preferred terminology), building on older work by the author \cite{dloa} and with Dunn \cite{iulg,sdl}, while working within the framework of canonical extensions \cite{mai-harding}  of lattice expansions.

This project was concluded in \cite{kata2z,duality2}, it developed in parallel with Gehrke's (with co-workers) generalized Kripke frames approach (RS frames) \cite{mai-gen} and the relations between the two approaches have been detailed in \cite{kata2z}. We note that, as far as the objectives of the current article are concerned, the RS-frames approach builds on Hartung's lattice representation \cite{hartung}, which inherits from Urquhart's lattice representation \cite{urq} an essential use of the axiom of choice. In addition,  full categorical dualities encounter difficulties in the RS-frames approach, relating to defining frame morphisms, which is what had already prevented either Urquhart's \cite{urq}, or Hartung's \cite{hartung} lattice representation to be extended to a full duality result. A brief review of attempts to define a suitable notion of morphism for polarities is presented by Goldblatt \cite[Introduction, pages 2--3]{goldblatt-morphisms2019}.

The present article focuses on lattices with a quasi-complementation operator and their associated not necessarily distributive logics with weak forms of negation.

Negation in logic has been an issue of intense study, and by a number of authors, including Do\v{s}en \cite{dosen,dosen2} and Vakarelov \cite{vak-neg-Rasiowa,vak-nelson}, where the modal interpretation of negation was introduced, Orlowska \cite{orlowska-negation,orlowska-negation-survey} where Urquhart's lattice representation is extended to the case of quasi-complemented lattices, as well as Dunn and Zhou \cite{dunn-zhou}, offering a classification of various notions of negation, including extended literature reviews and article collections on the subject, such as Dunn \cite{dunn-comparative} and Wansing \cite{wansing-negation}, respectively. Purely lattice-theoretic interest in lattices with a quasi-complementation operator is also demonstrated in the literature, with research focusing on particular cases of interest such as ortholattices (and orthologic \cite{goldb}), de Morgan, Ockham  and Kleene algebras \cite{ockham-alg}, or pseudo-complemented lattices \cite{pseudo-lee_1970}.

More recent and related to the present article is Almeida \cite{Almeida09}, working with RS-frames, and a brief discussion of how the two approaches relate is given in the conclusions section.

\subsection{Contribution}
The issues we address in this article are, first,  completeness  in relational semantics, as a consequence of representation.

Second,  canonicity of the logics and their associated algebras, which is established by the fact that the constructed second duals of the algebras are proven to be their canonical extensions.

Third, first-order correspondence facts are proven for the variously axiomatized systems we examine.

Fourth, we demonstrate that the J\'{o}nsson-Tarski framework, generalized to the non-distributive case in \cite{duality2}, applies uniformly to systems that may be non-distributive, or distributive, or even Boolean algebras. This is done by establishing a first-order frame condition such that the full complex algebra of Galois stable sets of any frame satisfying the condition is a completely distributive lattice.

Finally, we prove a full categorical duality result in which dependence on non-constructive principles is eliminated, and thus our results contribute to the Holliday-Bezhanishvili project \cite{choice-free-BA} of choice-free dualities.

The dualities in this article are obtained using the duality framework of \cite{duality2}, which covers the minimal system by instantiation of the general results and it is then extended to the case where additional axioms are considered. 

Canonicity is established, first by the fact that the lattice representation of \cite{sdl} delivers a canonical lattice extension, as proven by Gehrke and Harding \cite{mai-harding} and for operators (quasi-complements, in the present case) we use the proof in \cite[Section~3.3]{kata2z} that the extension of a normal lattice operator as we define it is, in the terminology of \cite{mai-harding}, its $\sigma$-extension if its output type is 1 (it returns joins) and it is its $\pi$-extension if its output type is $\partial$ (it returns meets). Correspondence results are calculated from scratch. In \cite{conradie-palmigiano}, a generalization of algorithmic Shalqvist correspondence to non-distributive logics has been presented, but the authors restrict it to frame semantics that are either RS-frames (based on Hartung's representation) or TiRS graphs \cite{craig} (based on  Plo\v{s}\v{c}ica's \cite{plo} restructuring of Urquhart's \cite{urq} lattice representation using partial homomorphisms). We believe that combining the results of this article with the modal translation \cite{redm} of substructural logics into sorted residuated modal logic an alternative approach to the Sahlqvist theory for non-distributive logics can be provided, but we have not yet focused on this.

\subsection{Structure}
The present article is structured as follows.
In Section \ref{section: quasi-comp lats} we present the algebras of interest, bounded lattices with a quasi-complementation operator, and their associated logics. We single-out some distinguished cases, allowing for both a distributive (notably De Morgan and Boolean algebras) and a non-distributive lattice base (such as involutive, or orthocomplement lattices).

Section \ref{frames section} starts with the necessary preliminaries (Section \ref{section: frames review}) on sorted frames with relations and generated operators, drawing on \cite{duality2}, and continues with Section \ref{section: frames for quasi comp lats} where frames for quasi-complemented lattices are discussed. A first-order axiomatic specification of the class of frames with respect to which the minimal logic with a weak negation operator can be shown to be sound is provided in Table~\ref{quasiframe axioms table}. The axiomatization is subsequently extended for each of the cases we consider.

Section \ref{rep section} presents choice-free representations of semilattices (Section \ref{section: sem rep}) and bounded lattices (Section \ref{lat rep section}) and concludes with Section \ref{normal ops rep section} detailing the representation of arbitrary normal lattice operators, drawing on \cite{duality2}.

In Section \ref{rep quasi comp section} we apply the representation framework of section \ref{normal ops rep section} to the particular case of quasi-complementation operators on bounded lattices. The results of this section establish that the varieties of quasi-complemented lattices we consider are closed under canonical extensions. Thereby, choice-free  completeness theorems via a canonical model construction can be proven for the logics of the algebraic structures considered.

Spectral duality theorems are proven in Section \ref{section: spectral duality}. The main result in this section is Theorem~\ref{nuM duality}, where we detail the duality between the categories $\mathbf{M}$ of bounded lattices with a minimal quasi-complementation operator and the category $\mathbf{SRF}^*_{\nu\mathrm{M}}$ of sorted residuated frames whose first-order axiomatization is given in Table~\ref{frame axioms nuM}. The remaining dualities are then obtained. In particular, Theorem~\ref{upper bound rel prop} provides a first-order frame condition for the lattice of Galois stable sets to be completely distributive, which is then used for the cases of representation and duality for De Morgan algebras and Boolean algebras.

\section{Quasi-complemented Lattices and Logics}
\label{section: quasi-comp lats}
Where $\mathbf{L}$ is a bounded lattice and $\{1,\partial\}$ is a 2-element set, we let $\mathbf{L}^1=\mathbf{L}$ and $\mathbf{L}^\partial=\mathbf{L}^{\mathrm{op}}$ (the opposite lattice). Extending established terminology \cite{jt1}, a function $f:\mathbf{L}_1\times\cdots\times\mathbf{L}_n\lra\mathbf{L}_{n+1}$, where each $\mathbf{L}_i$ is a bounded lattice, will be called {\em additive} and {\em normal}, or a {\em normal operator}, if it distributes over finite joins of the lattice $\mathbf{L}_i$, for each $i=1,\ldots n$, delivering a join in the lattice $\mathbf{L}_{n+1}$.

\begin{defn}\rm
\label{normal lattice op defn}
An $n$-ary operation $f$ on a bounded lattice $\mathbf L$ is {\em a normal lattice operator of distribution type  $\delta(f)=(i_1,\ldots,i_n;i_{n+1})\in\{1,\partial\}^{n+1}$}  if it is a normal additive function  $f:{\mathbf L}^{i_1}\times\cdots\times{\mathbf L}^{i_n}\lra{\mathbf L}^{i_{n+1}}$ (distributing over finite joins in each argument place), where  each $i_j$, for  $j=1,\ldots,n+1$,   is in the set $\{1,\partial\}$, hence ${\mathbf L}^{i_j}$ is either $\mathbf L$, or ${\mathbf L}^\partial$.

If $\tau$ is a tuple (sequence) of distribution types, a {\em normal lattice expansion of (similarity) type $\tau$} is a lattice with a normal lattice operator of distribution type $\delta$ for each $\delta$ in $\tau$.

The {\em category {\bf NLE}$_\tau$}, for a fixed similarity type $\tau$, has normal lattice expansions of type $\tau$ as objects. Its morphisms are the usual algebraic homomorphisms.
\end{defn}

In this article we focus on the class of lattices $\mathbf{L}=(L,\leq,\wedge,\vee,0,1,\nu)$ with a minimal quasi-complementation operator $\nu$, of increasing axiomatization strength, including at least the following:
\begin{tabbing}
\hskip1cm\=(antitonicity)\hskip1cm\= $a\leq b$ $\lra$ $\nu b\leq\nu a$\\
\>(normality) \> $\nu 0=1$\\
\>($\vee\wedge$)\> $\nu(a\vee b)\leq\nu a\wedge\nu b$.
\end{tabbing}
Given antitonicity, the operation $\nu$ satisfies the identity $\nu(a\vee b)=\nu a\wedge\nu b$, hence it is a normal lattice operator of distribution type $\delta(\nu)=(1;\partial)$.

\begin{rem}[Dual Minimal]\rm\label{dual minimal rem}
A treatment of the dual case, studied in \cite{dunn-zhou},
\begin{tabbing}
\hskip1cm\=(antitonicity)\hskip1cm\= $a\leq b$ $\lra$ $\nu b\leq\nu a$\\
\>(dual normality) \> $\nu 1=0$\\
\>($\wedge\vee$)\> $\nu(a\wedge b)\leq\nu a\vee\nu b$.
\end{tabbing}
follows easily from the results of this article (see, in particular Remark~\ref{unnecessity rem}, Figure~\ref{sigma-pi-canonical} and the proof of Theorem~\ref{fig 1 theorem}),  by observing that a dual minimal quasi-complement on a lattice $\mathbf{L}$ is a minimal quasi-complement on its dual $\mathbf{L}^\mathrm{op}$ (order reversed). Details for this case are left to the interested reader.  Note that given antitonicity the identity $\nu(a\wedge b)=\nu a\vee\nu b$ is derivable, hence the distribution type of $\nu$ in this case is $\delta(\nu)=(\partial;1)$.
\end{rem}
\begin{rem}[Involution]\rm\label{involution rem}
In the representation of a quasi-complemented lattice the difference between the cases of a minimal, or a dual minimal quasi-complement, following the representation framework of \cite{duality2}, is that in the first case, where the operator returns a meet, we represent the quasi-complement with its $\pi$-extension, whereas we use the $\sigma$-extension (in the sense of \cite{mai-harding}) for lattices with a dual-minimal quasi-complement (which returns a join). Both will be considered in the case of involutive lattices and proving that the representation of the lattice quasi-complement is an involution on the lattice of Galois stable sets reduces to verifying that the $\sigma$ and $\pi$-extensions of the lattice involution are identical.
\end{rem}

We list some basic facts about some quasi-complemented lattices of interest in Lemma \ref{neg basic facts lemma}.
\begin{lemma}\rm\label{neg basic facts lemma}
  Let $\mathbf{L}=(L,\leq,\wedge,\vee,0,1,\nu)$ be a bounded lattice with an antitone operation $\nu$.
  \begin{enumerate}
    \item $\nu$ forms a Galois connection on $\mathbf{L}$ ($a\leq\nu b$ iff $b\leq\nu a$) iff it satisfies the inequation $a\leq\nu\nu a$.
    \item If $a\leq\nu\nu a$ holds in the lattice, then the normality axiom $\nu 0=1$ and the identity $\nu(a\vee b)=\nu a\wedge\nu b$ are derivable.
    \item If $\nu\nu a\leq a$ holds in the lattice, then the identity $\nu(a\wedge b)=\nu a\vee\nu b$ is derivable.
    \item If either of the De Morgan identities $\nu(a\vee b)=\nu a\wedge\nu b$, or $\nu(a\wedge b)=\nu a\vee\nu b$, holds in the lattice, then antitonicity of $\nu$ is a derivable property.
    \item If $\nu$ is an involution ($a=\nu\nu a$) and, in addition,  the lattice is distributive, then it is a {\em De Morgan algebra}.
    \item If $\nu$ is an involution and, in addition, it satisfies the intuitionistic explosion principle (ex falso quidlibet) $a\wedge\nu a\leq 0$, then the lattice is an {\em Ortholattice} (orthocomplemented lattice).
    \item If $\nu$ is an involution satisfying the antilogism rule ($a\wedge b\leq c\lra a\wedge\nu c\leq\nu b$), then the lattice is a {\em Boolean algebra}.
  \end{enumerate}
\end{lemma}
\begin{proof}
Each of the claims (1) to (6) has a straightforward proof, left to the interested reader.   For (7), the hypothesis implies that $a\wedge b\leq c$ iff $a\wedge\nu c\leq\nu b$. This means that $\wedge$ is self-conjugate with respect to the involution $\nu$. To see that this implies distributivity, define $a\ra c=\nu(a\wedge\nu c)$ and observe that the conjugacy condition is equivalent to residuation of $\wedge$ and $\ra$, i.e. $a\wedge b\leq c$ iff $a\wedge\nu c\leq\nu b$ iff $b\leq a\ra c$. Distribution then follows from residuation. In addition, by part (2), $\nu 0=1$ holds and then also $\nu 1=\nu\nu 0=0$. Hence the intuitionistic principle $a\wedge\nu a\leq 0=\nu 1$ follows, since we can infer it from $a\wedge 1\leq a$ using antilogism. By the hypothesis that $\nu$ is an involution, the explosion principle $a\wedge\nu a\leq 0$ is equivalent to excluded middle $a\vee\nu a=1$. Hence the lattice is a Boolean algebra.
\end{proof}

Figure \ref{quasi-complements figure} summarizes the above results, where $\mathbb{DMA,O,INV}$ designate the equational classes (varieties) of De Morgan algebras, Ortholattices and lattices with an involution, respectively, $\mathbb{BA}$ designates the variety of Boolean algebras, the remaining two labels $\mathbb{M,G}$ designate the varieties of lattices with a minimal, or a Galois connected quasi-complementation operator, respectively, and the arrow label (dist) indicates addition of the distribution law $a\wedge(b\vee c)\leq(a\wedge b)\vee(a\wedge c)$. We use $\mathbb{M}^\partial, \mathbb{G}^\partial$ for the cases of the varieties of lattices with a dual minimal, or dual Galois, quasi-complement. A representation result for quasi-complemented lattices in the varieties $\mathbb{M}^\partial, \mathbb{G}^\partial$ can be easily obtained given the results in this article, and it is in fact indirectly given in the course of our presentation of the case of lattices with an involution. In the sequel, we focus on the varieties above $\mathbb{M}$ (included).

\begin{figure}[t]
  \caption{(Quasi)Complemented Lattices}\label{quasi-complements figure}
\mbox{}\\
\xymatrix{
&& \mathbb{BA}
\\
&\mathbb{DMA}\ar@{-}[ur]^{a\wedge\nu a=0} &
{{\frac{a\wedge b\leq c}{a\wedge\nu c\leq\nu b}}}\ar@{-}[u]
& \mathbb{O}\ar@{-}[ul]_{\mathrm{\!\!(dist)}}
\\
&&\mathbb{INV}\ar@{-}[ur]_{\! a\wedge\nu a=0}\ar@{-}[ul]^{\mathrm{(dist)}}
\ar@{-}[u]
\\
&&{\begin{array}{c}
    \mathbb{G}\\ 
    a\leq\nu\nu a
    \end{array}}
\ar@{-}[u]^{\nu\nu a\leq a}
&& {\begin{array}{c}
    \mathbb{G}^\partial\\
    \nu\nu a\leq a
    \end{array}}\ar@{.}[ull]_{a\leq\nu\nu a}
\\
&& {\begin{array}{c}
    \mathbb{M}\\
    \nu 0=1 \\
    \nu(a\vee b)=\nu a\wedge\nu b
  \end{array}}\ar@{-}[u]
&& {\begin{array}{c}
    \mathbb{M}^\partial\\
    \nu 1=0\\
    \nu(a\wedge b)=\nu a\vee\nu b
    \end{array}}\ar@{.}[u]
\\
&& a\leq b \lra \nu b\leq\nu a\ar@{.}[u]\ar@{.}[urr]
}
\vskip2mm
\hrulefill
\end{figure}

The propositional language $\mathcal{L}$ of quasi-complemented lattices is defined on a countable set $P$ of propositional variables by the schema
\[
\varphi:=p\;(p\in P)\midsp\top\midsp\bot\midsp\varphi\wedge\varphi\midsp\varphi\vee\varphi\midsp\neg\varphi.
\]
A valuation $v$ of the propositional language in an algebra $\mathbf{L}$ in the variety $\mathbb{M}$ (and any subvariety) is a map $v:P\lra L$. An interpretation $V$ of the propositional language is the unique homomorphism from the term algebra (the absolutely free algebra) on $P$ and into $\mathbf{L}$ extending $v$  to all sentences.

A formal equation is a string of the form $\varphi\approx\psi$ and a formal inequation $\varphi\preceq\psi$ is defined as usual by $\varphi\approx\varphi\wedge\psi$. A quasi-complemented lattice $\mathbf{L}$ validates a formal inequation, written $\varphi\models_\mathbf{L}\psi$ iff for any valuation $v$ we obtain $V(\varphi)\leq V(\psi)$.

A logic $\mathbf{\Lambda}$ in the language $\mathcal{L}$ is a set of pairs $(\varphi,\psi)$, which we write as $\varphi\proves_\mathbf{\Lambda}\psi$ and refer to them as (symmetric) sequents, containing the initial sequents and closed under the rules of Table~\ref{logic system} and under substitution. $\mathbf{M}$ designates the minimal logic in the above sense.

\begin{table}[!htbp]
\caption{Initial sequents and rules for the minimal logic $\mathbf{M}$}
\label{logic system}
\mbox{}\\[5mm]
\begin{tabular}{llllll}
$\varphi\proves\varphi$ & $\varphi\proves\top$ & $\bot\proves\varphi$ & $\infrule{\varphi\proves\psi\hskip4mm\psi\proves\vartheta}{\varphi\proves\vartheta}$ & $\infrule{\varphi\proves\psi}{\neg\psi\proves\neg\varphi}$\\[2mm]
$\varphi\wedge\psi\proves\varphi$ & $\varphi\wedge\psi\proves\psi$ & $\varphi\proves\varphi\vee\psi$ & $\psi\proves\varphi\vee\psi$ & $\neg(\varphi\vee\psi)\proves\neg\varphi\wedge\neg\psi$\\[2mm]
$\top\proves\neg\bot$
\end{tabular}

\hrulefill
\end{table}
A sequent $\varphi\proves\psi$ is valid in a quasi-complemented lattice  $\mathbf{L}$, written $\varphi\forces_\mathbf{L}\psi$, iff for any interpretation $v:P\lra L$, the inequation $V(\varphi)\leq V(\psi)$ holds in $\mathbf{L}$. A set of sequents $\Gamma$ (in particular, a logic $\mathbf{\Lambda}$) is valid in $\mathbf{L}$, written $\mathbf{L}\forces\Gamma$, if $\varphi\forces_\mathbf{L}\psi$ for every $(\varphi,\psi)\in\Gamma$. Evidently, the sequent $\varphi\proves\psi$ is valid iff the inequation $\varphi\preceq\psi$ is valid, in symbols $\varphi\forces_\mathbf{L}\psi$ iff $\varphi\models_\mathbf{L}\psi$.

A sequent $\varphi\proves\psi$ is valid in a subvariety $\mathbb{V}$ of $\mathbb{M}$ ($\mathbb{V}\in\texttt{Sub}\mathbb{M}$), written $\varphi\forces_\mathbb{V}\psi$, iff for any algebra $\mathbf{L}\in\mathbb{V}$ it holds that $\varphi\forces_\mathbf{L}\psi$. Validity of a set of sequents (in particular of a logic), or of a set of formal (in)equations, in  a subvariety $\mathbb{V}$ of $\mathbb{M}$ is defined analogously.

\begin{prop}\rm
The logic $\mathbf{M}$ is sound and complete in the variety $\mathbb{M}$.
\end{prop}
\begin{proof}
Soundness is immediate. For completeness we appeal to the Lindenbaum-Tarski construction, observing that (given the antitonicity rule) the relation $\varphi\equiv\psi$ iff $\varphi\proves\psi$ and $\psi\proves\varphi$ are both provable is a congruence. Details are straightforward.
\end{proof}

If $\Gamma$ is a set of sequents, then $\mathbf{M}+\Gamma$ designates the logic that results by adding members of $\Gamma$ as new initial sequents. Designate by $\texttt{Ext}\mathbf{M}$ the family of such extensions of $\mathbf{M}$ and observe that \texttt{Ext}$\mathbf{M}$ is a complete lattice under intersection.

Every extension $\mathbf{\Lambda}=\mathbf{M}+\Gamma$ determines a subvariety $\mathbb{V}\in\texttt{Sub}\mathbb{M}$ whose equational theory is $\mathcal{E}(\mathbb{V})=\mathcal{E}(\mathbb{M})\cup\{\varphi\preceq\psi\midsp (\varphi,\psi)\in\Gamma\}$.

\begin{defn}\rm
The maps $\mathrm{V}:\texttt{Ext}\mathbf{M}\leftrightarrows\texttt{Sub}\mathbb{M}:\mathrm{\Lambda}$ are defined on a logic $\mathbf{\Lambda}\in\texttt{Ext}\mathbf{M}$ and a variety $\mathbb{V}\in\texttt{Sub}\mathbb{M}$ by
\begin{align}
\mathrm{V}(\mathbf{\Lambda})&=\{\mathbf{L}\in\mathbb{M}\midsp\forall(\varphi,\psi)\in\mathbf{\Lambda}\;\varphi\forces_\mathbf{L}\psi\}\\
\mathrm{\Lambda}(\mathbb{V})&=\{(\varphi,\psi)\midsp\varphi\forces_\mathbb{V}\psi\}
\end{align}
\end{defn}

\begin{prop}[Definability]\rm
\label{definability}
For any subvariety $\mathbb{V}\in\texttt{Sub}\mathbb{M}$ and any quasi-complemented lattice $\mathbf{L}\in\mathbb{M}$, $\mathbf{L}\in\mathbb{V}$ iff $\mathbf{L}\forces\mathrm{\Lambda}(\mathbb{V})$. Equivalently, $\mathrm{V\Lambda}(\mathbb{V})=\mathbb{V}$.
\end{prop}
\begin{proof}
If $\mathbf{L}\in\mathbb{V}$, then clearly $\mathbf{L}\forces\mathrm{\Lambda}(\mathbb{V})$. Conversely, if $\mathbf{L}\forces\mathrm{\Lambda}(\mathbb{V})$ and $\varphi\preceq\psi$ is any inequation in the axiomatization of $\mathbb{V}$, then the fact that $\varphi\models_\mathbb{V}\psi$ iff $\varphi\forces_\mathbb{V}\psi$ implies that $\varphi\forces_\mathbf{L}\psi$ and so $\mathbf{L}$ validates the equational theory of $\mathbb{V}$, hence $\mathbf{L}\in\mathbb{V}$.
\end{proof}

\begin{prop}[Completeness]\rm
\label{completeness}
For every logic $\mathbf{\Lambda}\in\texttt{Ext}\mathbf{M}$ and every pair of sentences $(\varphi,\psi)$, $\varphi\proves\psi\in\mathbf{\Lambda}$ iff $\varphi\forces_{\mathrm{V}(\mathbf{\Lambda})}\psi$. Equivalently, $\mathrm{\Lambda V}(\mathbf{\Lambda})=\mathbf{\Lambda}$.
\end{prop}
\begin{proof}
If $\varphi\proves\psi$ is in $\mathbf{\Lambda}$, then by definition of $\mathrm{V}(\mathbf{\Lambda})$ it follows that $\varphi\forces_{\mathrm{V}(\mathbf{\Lambda})}\psi$.
Conversely, if $\mathbf{\Lambda}=\mathbf{M}+\Gamma$, then the equational axiomatization of its associated variety of quasi-complemented lattices is given by $\mathcal{E}=\mathcal{E}(\mathbf{M})\cup\{\varphi\preceq\psi\midsp (\varphi,\psi)\in\Gamma\}$. Thus $\mathbb{V}\in\mathrm{V}(\mathbf{\Lambda})$ iff $\mathbb{V}\models\mathcal{E}$. If $\varphi\forces_{\mathrm{V}(\mathbf{\Lambda})}\psi$, then in particular $\varphi\forces_{\mathbf{Lind}}\psi$, where $\mathbf{Lind}$ is the Lindenbaum-Tarski algebra of the logic $\mathbf{\Lambda}$, from which we infer $\varphi\proves\psi$ is in $\mathbf{\Lambda}$.
\end{proof}

\begin{prop}[Duality]\rm
\label{duality}
The maps $\mathrm{V,\Lambda}$ constitute a complete lattice dual isomorphism $\mathrm{V}:\texttt{Ext}\mathbf{M}\iso\texttt{Sub}\mathbb{M}^\mathrm{op}:\mathrm{\Lambda}$.
\end{prop}
\begin{proof}
Extensions of logics correspond to sub-quasivarieties and this accounts for the antitonicity of $\mathrm{V,\Lambda}$. By Propositions~\ref{definability} and~\ref{completeness}, $\mathrm{V,\Lambda}$ are inverses of each other, hence we have a bijection between $\texttt{Ext}\mathbf{M}$ and $\texttt{Sub}\mathbb{M}$.

Let $\mathbf{\Lambda}=\bigcap_{i\in I}\mathbf{\Lambda}_i=\bigcap_{i\in I}(\mathbf{M}+\Gamma_i)$ and $\mathbb{V}_i=\mathrm{V}(\mathbf{\Lambda}_i)$. The equational axiomatization of $\mathbb{V}_i$ is then $\mathcal{E}(\mathbb{V}_i)=\mathcal{E}(\mathbf{M})\cup\{\varphi\preceq\psi\midsp(\varphi,\psi)\in\Gamma_i\}$.

Since $\bigvee_{i\in I}\mathbb{V}_i=\{\mathbf{L}\in\mathbb{M}\midsp\mathbf{L}\models\bigcap_{i\in I}\mathcal{E}(\mathbb{V}_i)\}$ it follows that $\mathbf{L}\in\bigvee_{i\in I}\mathbb{V}_i$ iff $\mathbf{L}\models\mathcal{E}(\mathbb{M})\cup\bigcap_{i\in I}\mathcal{E}(\mathbb{V}_i)$ iff $\varphi\models_\mathbf{L}\psi$ for any $\varphi\preceq\psi$ in $\mathcal{E}=\mathcal{E}(\mathbb{M})\cup\bigcap_{i\in I}\mathcal{E}(\mathbb{V}_i)$. Given that $\mathcal{E}=\bigcap_{i\in I}(\mathcal{E}(\mathbb{M})\cup\mathcal{E}(\mathbb{V}_i))$ we have, equivalently, that $\mathbf{L}\in\bigvee_{i\in I}\mathbb{V}_i$ iff $\mathbf{L}\forces\bigcap_{i\in I}(\mathbf{M}+\Gamma_i)$ iff $\mathbf{L}\forces\bigcap_{i\in I}\mathbf{\Lambda_i}$. This shows that $\mathrm{V}\left(\bigcap_{i\in I}\mathbf{\Lambda}_i\right)=\bigvee_{i\in I}\mathrm{V}(\mathbf{\Lambda}_i)$.
\end{proof}

This dual isomorphism allows us to switch from logics to varieties of quasi-complemented lattices and back without any loss of information.

\section{Sorted Residuated Frames (SRFs)}
\label{frames section}
Relational semantics for logics possibly without distribution rely on respective representation theorems, in the sense that for an algebraizable logic completeness is established by a canonical model construction, the underlying frame of which is preciseley the frame constructed in the related representation of its Lindenbaum-Tarski algebra. 

The first bounded lattice representation result has been given by Urquhart \cite{urq}, using doubly-ordered Stone spaces, and where the points of the canonical frame are maximally disjoint filter-ideal pairs.  Sorted frames have been employed in both Hartung's \cite{hartung} lattice representation and in the representation due to this author and Dunn \cite{iulg,sdl}.  In either case, lattices are identified as sublattices of the complete lattice of {\em stable sets}, which are exactly the fixpoints of the closure operator generated by composition of the two maps of a Galois connection defined on the frame.

Some basic preliminaries on sorted frames are necessary, detailed in the next section, before we can define (sorted) frames for quasi-complemented lattices. In this article, since no other kind of frame is involved, we use `sorted frame', `polarity', `sorted residuated frame', or just `frame' interchangeably.

\subsection{Sorted Frame Preliminaries}
\label{section: frames review}
Regard $\{1,\partial\}$ as a set of sorts and let $Z=(Z_1,Z_\partial)$ be a sorted set.
Sorted residuated frames $\mathfrak{F}=(Z_1,\upv,Z_\partial)$ are triples consisting of nonempty sets $Z_1=X,Z_\partial=Y$ and a binary relation ${\upv}\subseteq X\times Y$.

The relation $\upv$ will be referred to as the {\em Galois relation} of the frame. It generates a Galois connection $(\;)\rperp:\powerset(X)\leftrightarrows\powerset(Y)^\partial:\lperp(\;)$ ($V\subseteq U\rperp$ iff $U\subseteq\lperp V$)
\begin{tabbing}
\hskip2cm\=$U\rperp$\hskip2mm\==\hskip1mm\=$\{y\in Y\midsp\forall x\in U\; x\upv y\}$ \hskip1mm\==\hskip1mm\= $\{y\in Y\midsp U\upv y\}$\\
\>$\lperp V$\>=\>$\{x\in X\midsp \forall y\in V\;x\upv y\}$\>=\>$\{x\in X\midsp x\upv V\}$.
\end{tabbing}
We will also have use for the complement $I$ of the Galois relation $\upv$ and we will designate frames using either the Galois relation $\upv$, or its complement $I$. A subset $A\subseteq X$ will be called {\em stable} if $A={}\rperp(A\rperp)$. Similarly, a subset $B\subseteq Y$ will be called {\em co-stable} if $B=({}\rperp B)\rperp$. Stable and co-stable sets will be referred to as {\em Galois sets}, disambiguating to {\em Galois stable} or {\em Galois co-stable} when needed and as appropriate. By $\gpsi,\gphi$ we designate the complete lattices of stable and co-stable sets, respectively. Note that the Galois connection restricts to a dual isomorphism $(\;)\rperp:\gpsi\iso\gphi^\partial:{}\rperp(\;)$.

Structures $(X,\upv,Y)$ are referred to as {\em polarities}, following Birkhoff \cite{birkhoff}, also as {\em formal contexts} in Formal Concept Analysis (FCA) \cite{wille2} and we use the standard FCA priming notation for each of the two Galois maps ${}\rperp(\;),(\;)\rperp$. This allows for stating and proving results for each of $\gpsi,\gphi$ without either repeating definitions and proofs, or making constant appeals to duality. Thus for a Galois set $G$, $G'=G\rperp$, if $G\in\gpsi$ ($G$ is a Galois stable set), and otherwise $G'={}\rperp G$, if $G\in\gphi$ ($G$ is a Galois co-stable set).

For an element $u$ in either $X$ or $Y$ and a subset $W$, respectively of $Y$ or $X$, we write $u|W$, under a well-sorting assumption, to stand for either $u\upv W$ (which stands for $u\upv w$, for all $w\in W$), or $W\upv u$ (which stands for $w\upv u$, for all $w\in W$), where well-sorting means that either $u\in X, W\subseteq Y$, or $W\subseteq X$ and $u\in Y$, respectively. Similarly for the notation $u|v$, where $u,v$ are elements of different sort.

Preorder relations are induced on each of the sorts, by setting for $x,z\in X$, $x\preceq z$ iff $\{x\}\rperp\subseteq\{z\}\rperp$ and, similarly, for $y,v\in Y$, $y\preceq v$ iff ${}\rperp\{y\}\subseteq{}\rperp\{v\}$.
A (sorted) frame is called {\em separated} if the preorders $\preceq$ (on $X$ and on $Y$) are in fact partial orders $\leq$.

We use $\Gamma$ to designate  upper closure  $\Gamma U=\{z\in X\midsp\exists x\in U\;x\preceq z\}$, for $U\subseteq X$, and similarly for $U\subseteq Y$. The set $U$ is {\em increasing} (an upset) iff $U=\Gamma U$. For a singleton set $\{x\}\subseteq X$ we write $\Gamma x$, rather than $\Gamma(\{x\})$ and similarly for $\{y\}\subseteq Y$.

\begin{lemma}[\hskip-0.1mm{\cite[Lemma~3.3]{duality2}}]\rm
\label{basic facts}
Let $\mathfrak{F}=(X,\upv,Y)$ be a  polarity and $u$ a point in $Z=X\cup Y$.
\begin{enumerate}
\item $\upv$ is increasing in each argument place (and thereby its complement $I$ is decreasing in each argument place).
\item $(\Gamma u)'=\{u\}'$, while $\Gamma u=\{u\}^{\prime\prime}$ is a Galois set.
\item Galois sets are increasing, i.e. $u\in G$ implies $\Gamma u\subseteq G$.
\item For a Galois set $G$, $G=\bigcup_{u\in G}\Gamma u$.
\item For a Galois set $G$, $G=\bigvee_{u\in G}\Gamma u=\bigcap_{v|G}\{v\}'$.
\item For a Galois set $G$ and any set $W$, $W^{\prime\prime}\subseteq G$ iff $W\subseteq G$.\telos
\end{enumerate}
\end{lemma}

It is typical in the context of canonical extensions of lattices to refer to principal upper sets $\Gamma x\in\gpsi\; (x\in X= \filt(\mathbf{L}))$, as {\em filter}, or {\em closed elements} of $\gpsi$ and to sets ${}\rperp\{y\}\in\gpsi\; (y\in Y=\idl(\mathbf{L}))$ as {\em open}, or {\em ideal} elements of $\gpsi$, and similarly for sets $\Gamma y, \{x\}\rperp$ with $x\in X, y\in Y$. This risks creating an unfortunate confusion with topological terminology and we shall have to rely on context to disambiguate. In addition, we shall always use ``closed set'' and ``open set'' when topologically closed/open sets are meant. Furthermore, a closed element $\Gamma x$ is said to be {\em clopen} if $\Gamma x={}\rperp\{y\}$ for some $y\in Y$, which is unique when the frame is separated.

By Lemma \ref{basic facts}, the closed elements of $\gpsi$  join-generate $\gpsi$, while the open elements meet-generate $\gpsi$ (similarly for $\gphi$).

For a sorted relation $R\subseteq\prod_{j=1}^{j=n+1}Z_{i_j}$, where $i_j\in\{1,\partial\}$ for each $j$ (and thus $Z_{i_j}=X$ if $i_j=1$ and $Z_{i_j}=Y$ when $i_j=\partial$), we make the convention to regard it as a relation $R\subseteq Z_{i_{n+1}}\times\prod_{j=1}^{j=n}Z_{i_j}$, we agree to write its sort type as $\sigma(R)=(i_{n+1};i_1\cdots i_n)$ and for a tuple of points of suitable sort we write $uRu_1\cdots u_n$ for $(u,u_1,\ldots,u_n)\in R$.

\begin{defn}[Galois dual relation] \label{Galois dual relations}\rm
For an $(n+1)$-ary sorted relation $R$ its {\em Galois dual relation} $R'$ is defined by $uR'v_1\cdots v_n$ iff $\forall w\;(wRv_1\cdots v_n\lra w|u)$. In other words, $R'\vec{v}=(R\vec{v})'$, letting $\vec{v}=v_1\cdots v_n$.
\end{defn}

\begin{defn}[Sections of relations]\rm
\label{sections defn}
For an $(n+1)$-ary relation $R^\sigma$ (of sort $\sigma$) and an $n$-tuple $\vec{u}$, $R^\sigma\vec{u}=\{w\midsp wR^\sigma\vec{u}\}$ is the {\em section} of $R^\sigma$ determined by $\vec{u}$. To designate a section of the relation at the $k$-th argument place we let $\vec{u}[\_]_k$ be the tuple with a hole at the $k$-th argument place. Then $wR^\sigma\vec{u}[\_]_k=\{v\midsp wR^\sigma\vec{u}[v]_k\}\subseteq Z_{i_k}$ is the $k$-th section of $R^\sigma$.
\end{defn}

\begin{thm}\rm
\label{upper bound rel prop}
Let $\mathfrak{F}=(X,\upv,Y)$ be a sorted frame (a polarity) and $\gpsi$ the complete lattice of stable sets. Let $R_\leq$ be the ternary upper bound relation on $X$ defined by $uR_\leq xz$ iff both $x\preceq u$ and $z\preceq u$.  If all sections of the Galois dual relation $R'_\leq$ of $R_\leq$ are Galois sets, then $\gpsi$ is completely distributive.
\end{thm}
\begin{proof}
Let $\alpha_R$ be the image operator generated by $R_\leq$, $\alpha_R(U,W)=\bigcup_{u\in U}^{w\in W}Ruw$. Notice that, for stable sets $A,C$ (more generally, for increasing sets), $\alpha_R(A,C)=A\cap C$. Hence $\overline{\alpha}_R(A,C)=\alpha_R(A,C)=A\cap C$, since Galois sets are closed under intersection. Given the section stability hypothesis for the Galois dual relation $R'_\leq$ of $R_\leq$,  \cite[Theorem~3.12]{duality2} applies, from which distribution of $\overline{\alpha}_R$ (i.e. of intersection) over arbitrary joins of stable sets is concluded.
\end{proof}

\subsection{Frames for Quasi-Complemented Lattices}
\label{section: frames for quasi comp lats}
Frames for quasi-complemented lattices, with at least a minimal quasi complement, are structures $\mathfrak{F}=(X,\upv,Y,S_\vee)$ with $\sigma(S_\vee)=(\partial;1)$, i.e. $S_\vee\subseteq Y\times X$. Note that in the case of a dual minimal quasi-complement we need to consider a binary relation $R_\wedge\subseteq X\times Y$ (of the dual sort type, comparing to $S_\vee$).

We list in Table \ref{quasiframe axioms table}, after \cite[Table~2]{duality2}, the minimal axiomatization we assume, which will be strengthened in the sequel imposing, for duality purposes, a spectral topology on each of $X,Y$. The only difference with the frame axioms of \cite{duality2} is that we add axiom (F0) which ensures that the empty set is (co)stable and, more precisely, ${}\rperp Y=\emptyset$ and similarly $X\rperp=\emptyset$.

Note, in particular, that axioms (F1) and (F2) imply that there is a (sorted) function $\widehat{\nu}_S:X\lra Y$ on the points of the frame such that, dropping the subscript, for simplicity, $\widehat{\nu}(x)=y$ iff $S_\vee x=\Gamma y$. Therefore
\begin{equation}\label{alternative-def-of-S}
yS_\vee x\;\mbox{ iff }\; \widehat{\nu}(x)\leq y
\end{equation}

\begin{table}[t]
\caption{Axioms for sorted frames $\mathfrak{F}=(X,\upv,Y,S_\vee)$}
\label{quasiframe axioms table}
\begin{enumerate}
  \item[(F0)] The complement $I$ of the Galois relation $\upv$ of the frame is quasi-serial, i.e. the conditions
  $\forall x\in X\exists y\in Y\; xIy$ and $\forall y\in Y\exists x\in X\; xIy$ hold.
  \item[(F1)] The frame is separated.
  \item[(F2)] For any $x\in X$, the section $S_\vee x$ of the frame relation $S_\vee$ is a closed element of $\gphi$.
  \item[(F3)] For any $y\in Y$, the section $yS_\vee$ of the frame relation $S_\vee$ is decreasing (a downset).
  \item[(F4)] Both sections of the Galois dual relation $S'_\vee\subseteq X\times X$ of $S_\vee$ are Galois sets.
\end{enumerate}
\hrulefill
\end{table}

\begin{defn}\label{image-op-and-bot}\rm
The {\em incompatibility relation} ${\perp}\subseteq X\times X$ is defined as the Galois dual relation $S'_\vee$ of the frame relation $S_\vee$. More explicitly, for all $x,z\in X$,
\begin{equation}\label{incompatibility-def}
x\perp z\mbox{ iff }x S'_\vee z\mbox{ iff }\forall y\in Y(yS_\vee z\lra x\upv y)
\end{equation}
Furthermore, for any set $U\subseteq X$, define a {\em star set-operator} by
\begin{equation}\label{star-op}
  U^*=\{x\in X\midsp x\perp U\}=\{x\in X\midsp\forall u\in U\;x\perp u\}
\end{equation}
\end{defn}

Given a sorted frame $\mathfrak{F}=(X,\upv,Y,S_\vee)$ as above, the language of quasi-complemented lattices can be interpreted in the frame. If $v:P\lra\gpsi$ is an assignment of stable sets to propositional variables, an interpretation $\val{\;}$ and a co-interpretation $\yvval{\;}$ are functions required to satisfy the following mutual recursion clauses, as  well as the identities $\yvval{\varphi}=\val{\varphi}\rperp$ and then also $\val{\varphi}={}\rperp\yvval{\varphi}$, for any sentence $\varphi$.

\begin{tabbing}
\hskip4mm\= $\val{p}$\hskip8mm\= = \hskip2mm\= $v(p)$\hskip4cm\= $\yvval{p}$\hskip8mm\= = \hskip2mm\=$v(p)'=v(p)\rperp$\\
\>$\val{\top}$\>=\> $X$\> $\yvval{\perp}$\>=\> $Y$\\
\>$\val{\varphi\wedge\psi}$\>=\> $\val{\varphi}\cap\val{\psi}$\> $\yvval{\varphi\vee\psi}$\>=\> $\yvval{\varphi}\cap\yvval{\psi}$\\
\> $\val{\neg\varphi}$\>=\> $\{x\in X\midsp x\perp\val{\varphi}\}=\val{\varphi}^*$
\end{tabbing}
Satisfaction $\forces$ and co-satisfaction (refutation) $\dforces$ relations can be then defined as usual, by $x\forces\varphi$ iff $x\in\val{\varphi}$ and $y\dforces\varphi$ iff $y\in\yvval{\varphi}$ iff $\forall x(x\forces\varphi\lra x\upv y)$, where $x\in X$ and $y\in Y$. Soundness and completeness issues will be discussed after establishing the basic properties of frames for quasi-complemented lattices.

In Corollary~\ref{star-codistrib-coro} we conclude the proof that the star-operator of equation \eqref{star-op} co-distributes over arbitrary joins of stable sets, i.e. that $\left(\bigvee_{j\in J}A_j\right)^*=\bigcap_{j\in J}A^*_j$, and this implies that $\gpsi$ is a quasi-complemented lattice, hence the minimal logic $\mathbf{M}$ is sound in the class of frames we consider. To prove co-distribution of the star operator over joins we proceed by first providing an alternative construction of the star-operator, using the image operator associated to the frame relation $S_\vee$.

Let $\eta_S:\powerset(X)\lra\powerset(Y)$ be the sorted image operator generated by $S_\vee$, defined on $U\subseteq X$  by
\begin{equation}\label{sorted image op}
\eta_S(U)=\{y\in Y\midsp\exists x\in X(yS_\vee x\mbox{ and }x\in U)\}=\bigcup_{x\in U}S_\vee x,
\end{equation}
and $\overline{\eta}_S:\gpsi\lra\gphi$ be the closure of its restriction to Galois sets.
 By axiom (F2), $S_\vee x$ is a Galois set.

\begin{lemma}\label{s-nu-gamma}\rm
For any point $x\in X$, $S_\vee x=\overline{\eta}_S(\Gamma x)=\eta_S(\Gamma x)=\Gamma(\widehat{\nu}(x))$.
\end{lemma}
\begin{proof}
By definition, $\eta_S(\Gamma x)=\bigcup_{z\in\Gamma x}S_\vee z=\bigcup_{x\leq z}S_\vee z$. It follows from axiom (F3) that for $x\leq z$ we have $S_\vee z\subseteq S_\vee x$, hence $\bigcup_{x\leq z}S_\vee z=S_\vee x$, a closed element by axiom (F2), hence $S_\vee x=\eta_S(\Gamma x)=\overline{\eta}_S(\Gamma x)$. Observing that we also have $\widehat{\nu}(x)=y$ iff $S_\vee x=\Gamma y$, the proof is complete.
\end{proof}

 It follows from Lemma~\ref{s-nu-gamma} that for a stable set $A\in\gpsi$,
\begin{equation}\label{closure of sorted def}
\overline{\eta}_S(A)=\left(\bigcup_{x\in A}S_\vee x\right)''=\bigvee_{x\in A}S_\vee x=\bigvee_{x\in A}\overline{\eta}_S(\Gamma x)\in\gphi.
\end{equation}

\begin{prop}\label{eta-S-distrib-property}\rm
$\overline{\eta}_S$ distributes over arbitrary joins in the lattice $\gpsi$, returning a join in $\gphi$. In symbols, $\overline{\eta}_S\left(\bigvee_{j\in J}A_j\right)=\bigvee_{j\in J}\overline{\eta}_S(A_j)$.
\end{prop}
\begin{proof}
In \cite[Theorem~3.12]{duality2} we proved that if all sections of the Galois dual relation $R'$ of a frame relation $R$ are Galois sets, then the closure $\overline{\alpha}_R$ of the restriction to Galois sets of the sorted image operator $\alpha_R$ generated by the frame relation $R$ distributes over arbitrary joins in each argument place. The current proposition is just an instance of Theorem~3.12 of \cite{duality2}, given that by the frame axioms of Table~\ref{quasiframe axioms table}, both sections of $S_\vee$ are stable, hence $\overline{\eta}_S$ has the claimed distribution property.
\end{proof}

Define now
\begin{equation}
\overline{\eta}_\vee(A)=(\overline{\eta}_S(A))'=\bigcap_{z\in A}S'_\vee z \label{def eta nu}
\end{equation}
so that $\overline{\eta}_\vee$ is a single-sorted operation (on $\gpsi$) derived from $\overline{\eta}_S$ by composition with the Galois connection.

\begin{lemma}\label{ast-is-eta-nu}\rm
For any stable set $A\in\gpsi$, $A^*=\overline{\eta}_\vee(A)$.
\end{lemma}
\begin{proof}
Given that $\overline{\eta}_\vee(A)=\bigcap_{z\in A}S'_\vee z$ and using the definition of the incompatibility relation $\perp$, $x\in\overline{\eta}_\vee(A)$ iff for all $z\in A$, $x\perp z$ iff $x\perp A$ iff $x\in A^*$.
\end{proof}
The following is then immediate, given the definition of $\overline{\eta}_\vee$ together with the fact that the Galois connection induced by $\upv$ restricts to a dual equivalence of the lattices of stable and co-stable sets.

\begin{coro}\label{star-codistrib-coro}\rm
The star operation co-distributes over arbitrary joins of stable sets. In symbols $\left(\bigvee_{j\in J}A_j\right)^*=\bigcap_{j\in J}A^*_j$.\telos
\end{coro}

Corollary~\ref{star-codistrib-coro} establishes that if $\mathfrak{F}=(X,\upv,Y,S_\vee)$ is a sorted frame validating the frame axioms in Table~\ref{quasiframe axioms table}, then its full complex algebra of stable sets $\mathfrak{F}^+=(\gpsi,\subseteq,\bigcap,\bigvee,\emptyset,X,(\;)^*)$ is in the variety $\mathbb{M}$ of quasi-complemented lattices with a minimal quasi-complementation operator.

\begin{rem}[Frames for Dual Minimal Quasi-complements]\rm 
\label{unnecessity rem}
The frames to be considered in this case are structures $\mathfrak{F}=(X,\upv,Y,R_\wedge)$ satisfying axioms (F0)--(F4) of Table \ref{quasiframe axioms table}, with $R_\wedge\subseteq X\times Y$ in place of $S_\vee\subseteq Y\times X$. In other words, we assume in the frame axiomatization for this case that for all $y$, the set $R_\wedge y$ is a closed element in $\gpsi$ (axiom (F2)) and designate its generating point by $\widetilde{\nu}(y)\in X$, for each $y\in Y$. By axiom (F3), for each $x\in X$, the section $xR_\wedge$ is a downset. Following our notational convention, we designate its Galois dual relation by $R'_\wedge\subseteq Y\times Y$. Since $R_\wedge y=\Gamma\widetilde{\nu}(y)\in\gpsi$, we obtain that $R'_\wedge y=\{\widetilde{\nu}(y)\}\rperp\in\gphi$. By axiom (F4), for each $v\in Y$ the section $vR'_\wedge$ is a Galois (co-stable) set in $\gphi$.

Let $\eta_R:\powerset(Y)\lra\powerset(X)$ be the sorted image operator generated by the (sorted) relation $R_\wedge$
\begin{equation}\label{eta-R-defn}
\eta_R(V)=\{x\in X\midsp\exists y\in Y(xR_\wedge y\wedge y\in V)\}=\bigcup_{y\in V}R_\wedge y
\end{equation}
and let $\overline{\eta}_R:\gphi\lra\gpsi$ be the closure of its restriction to Galois (co-stable) sets, defined by $\overline{\eta}_R(B)=\bigvee_{y\in B}R_\wedge y$, and set $\overline{\eta}_\wedge(A)=\overline{\eta}_R(A')=\bigvee_{A\upv y}R_\wedge y$.

By the same arguments as for $\overline{\eta}_S$ (see Lemma~\ref{s-nu-gamma} and Proposition~\ref{eta-S-distrib-property}) we obtain that for any $y\in Y$
\begin{equation}\label{eta-R-Gamma-y}
\overline{\eta}_R(\Gamma y)=R_\wedge y=\eta_R(\Gamma y)=\Gamma\widetilde{\nu}(y)\in\gpsi
\end{equation}
and that for $B_i\in\gphi$, for each $i\in I$,
\begin{equation}\label{eta-R-dist}
\overline{\eta}_R(\bigvee_{i\in I}B_i)=\bigvee_{i\in I}\overline{\eta}_R(B_i)\in\gpsi. 
\end{equation}
Hence, 
\begin{equation}\label{eta-wedge-codist}
\overline{\eta}_\wedge(\bigcap_{i\in I}A_i)=\overline{\eta}_R((\bigcap_{i\in I}A_i)')=
\overline{\eta}_R(\bigvee_{i\in I}A'_i)=\bigvee_{i\in I}\overline{\eta}_R(A'_i)=\bigvee_{i\in I}\overline{\eta}_\wedge(A_i),
\end{equation}
i.e. $\overline{\eta}_\wedge$ co-distributes over arbitrary meets of stable sets (turning them to joins), so that it is a dual minimal operator in the lattice $\gpsi$ of Galois stable sets.
\end{rem}

Next, we turn to establishing frame conditions under which $\mathfrak{F}^+$ lies within one of the subvarieties $\mathbb{G, INV, DM, O}$ or $\mathbb{BA}$ of Figure~\ref{quasi-complements figure}. We leave the (similar) cases for $\mathbb{M}^\partial,\mathbb{G}^\partial$ to the interested reader. We conclude with the necessary correspondence conditions in Corollary~\ref{frame conditions for complex algebra coro}, which follows from Lemmas~\ref{galois frame lemma}, \ref{ex falso frame lemma} and \ref{involution frame lemma}. The proofs of the Lemmas use the fact that $\overline{\eta}_S$ is residuated with a map $\overline{\zeta}_S$ and providing a suitable definition of the residual in Proposition~\ref{residual-of-eta-s} allows us to conclude in Corollary~\ref{frame conditions for complex algebra coro} the intended correspondence results. Note that having $\overline{\eta}_S\dashv\overline{\zeta}_S$ and composing with the Galois connection, as shown in Figure~\ref{eta-zeta-figure}, we obtain a Galois connected pair of maps $\overline{\eta}_\vee\dashv\overline{\zeta}_\vee$.

\begin{figure}[t]
\caption{The maps $\overline{\eta}_S,\overline{\eta}_\nu$ and their respective adjoints $\overline{\zeta}_S,\overline{\zeta}_\vee$}
\label{eta-zeta-figure}
\[
\xymatrix{
\powerset(X)\ar@<0.5ex>[rr]^{\eta_S} && \powerset(Y)\ar@<0.5ex>[ll]^{\zeta_S} & \eta_S\dashv\zeta_S 
\\
\mathcal{G}(X)\ar@<0.5ex>[rr]^{\overline{\eta}_S}\ar@<0.5ex>[d]^{\overline{\eta}_\vee=(\;)^*}
&& \mathcal{G}(Y)\ar@<0.5ex>[dll]\ar@<0.5ex>[ll]^{\overline{\zeta}_S} & \overline{\eta}_S\dashv\overline{\zeta}_S
\\
\mathcal{G}(X)^{\mathrm{op}}\ar@<0.5ex>[u]^{(\;)^\smalltriangleup=\overline{\zeta}_\vee} \ar@<0.5ex>[urr]_{(\;)'}
&&& A\subseteq C^\smalltriangleup\mbox{ iff }C\subseteq A^*
}
\]
\hrulefill
\end{figure}

\begin{prop}\label{residual-of-eta-s}\rm
The map $\overline{\zeta}_S:\gphi\lra\gpsi$ defined on a co-stable set $B\in\gphi$ by $\overline{\zeta}_S(B)=\{x\in X\midsp\overline{\eta}_S(\Gamma x)\subseteq B\}$ is the right residual of $\overline{\eta}_S$.
\end{prop}
\begin{proof}
Since $\overline{\eta}_S:\gpsi\lra\gphi$ distributes over arbitrary joins, by Proposition~\ref{eta-S-distrib-property}, it is residuated with a map $\overline{\zeta}_S:\gphi\lra\gpsi$, which maps meets of $\gphi$ to meets of $\gpsi$,
 defined on $B\in\gphi$  by
\begin{equation}\label{standard residual defn}
\overline{\zeta}_S(B)=\bigvee\{A\in\gpsi\midsp\overline{\eta}_S(A)\subseteq B\}.
\end{equation}
Instantiating \cite[Theorem~3.14]{duality2} to the present case, we obtain that the residual of $\overline{\eta}_S$ is the restriction to Galois sets of the residual $\zeta_S$ of the image operator $\eta_S$, explicitly defined by $\overline{\zeta}_S(B)=\bigcup\{A\in\gpsi\midsp\overline{\eta}_S(A)\subseteq B\}$, so that the join in the standard definition~\eqref{standard residual defn} of the residual $\overline{\zeta}_S$ is actually set-theoretic union. Thereafter, instantiating
 \cite[Lemma~3.15]{duality2} to the present case, we obtain that $\overline{\zeta}_S(B)$ is equivalently defined by equation \eqref{beta equiv 2 special},
\begin{equation}
\label{beta equiv 2 special}
\overline{\zeta}_S(B)=\{x\in X\midsp\overline{\eta}_S(\Gamma x)\subseteq B\}
\end{equation}
and this completes the proof.
\end{proof}

By duality of $\gpsi$ and $\gphi$, every $B\in\gphi$ is $B=C'$ for some $C\in\gpsi$. Hence we obtain that $\overline{\eta}_S(A)\subseteq C'$ iff $A\subseteq\overline{\zeta}_S(C')$. From this, setting
\begin{equation}
\overline{\zeta}_\vee(C)=\overline{\zeta}_S(C') \label{def zeta nu}
\end{equation}
we obtain the Galois connection condition $C\subseteq\overline{\eta}_\vee(A)$ iff $A\subseteq\overline{\zeta}_\vee(C)$. Recalling that  $A^*=\overline{\eta}_\vee(A)$, from Lemma~\ref{ast-is-eta-nu}, and setting
\begin{equation}\label{def triangleup}
A^{\smalltriangleup}=\overline{\zeta}_\vee(A)
\end{equation}
we can rewrite the Galois connection condition as 
\begin{equation}\label{eta-zeta galois}
A\subseteq C^{\smalltriangleup}\;\mbox{ iff }\; C\subseteq A^*.
\end{equation}

\begin{lemma}
\rm \label{galois frame lemma}
Let $\mathfrak{F}=(X,\upv,Y,S_\vee)$ be a frame subject to the axioms of Table \ref{quasiframe axioms table} and let $A\in\gpsi$ be any stable set. The following are equivalent
    \begin{enumerate}
      \item $\perp$ is symmetric
      \item $A\subseteq A^{**}$
      \item $A^*=A^{\smalltriangleup}$.
    \end{enumerate}

\end{lemma}
\begin{proof}
For 
(1)$\Ra$(2),
suppose, for a contradiction, that $x\in A$, but $x\not\in A^{**}$. Let  $z\in A^*=\{z\midsp\forall u(u\in A\lra z\perp u)\}$ such that $x\perp z$ fails. Since $x\in A$,  $z\perp x$ holds. By  symmetry of $\perp$ it follows $x\perp z$, contradiction. Hence $A\subseteq A^{**}$, for any $A\in\gpsi$.

For (2)$\Ra$(3),
from the hypothesis $A\subseteq A^{**}$ and Lemma~\ref{neg basic facts lemma}~(1) it is concluded that $(\;)^*$ forms a Galois connection with itself. By Lemma~\ref{ast-is-eta-nu}, $A^*=\overline{\eta}_\vee(A)$, where by definition $\overline{\eta}_\vee(A)=(\overline{\eta}_S(A))'$. In Proposition~\ref{residual-of-eta-s} we defined a map $\overline{\zeta}_S$, proven to be the right residual of $\overline{\eta}_S$. We then defined $\overline{\zeta}_\vee(A)=\overline{\zeta}_S(A')$, hence obtaining a Galois connection with the two maps $\overline{\eta}_\vee$ and $\overline{\zeta}_\vee$. Recall that we simplified notation setting $A^\smalltriangleup=\overline{\zeta}_\vee(A)$. But now $(\;)^*$ forms a Galois connection both with itself and with $(\;)^\smalltriangleup$, hence by uniqueness of adjoints $(\;)^\smalltriangleup=(\;)^*$ follows.

For (3)$\Ra$(2),
by definition $(\;)^{\smalltriangleup}$ is Galois connected with $(\;)^*$ and hence if the two are equal, then the result follows by using Lemma \ref{neg basic facts lemma}.

For (2)$\Ra$(1), recall that $\perp=S'_\vee$, both sections of which are stable sets, which are increasing sets by Lemma \ref{basic facts}, and that by Lemma  \ref{neg basic facts lemma} the hypothesis means that $(\;)^*$ is a Galois connection on the lattice of stable sets.  Assuming $x\perp z$, we then have $x\perp\Gamma z$, which means that $x\in(\Gamma z)^*=\{x\midsp\forall u(z\leq u\lra x\perp u)\}$. If $x\leq w$ and $z\leq u$, then by $x\perp z$ we also have $w\perp u$ and this means that $\Gamma x\subseteq(\Gamma z)^*$. Since $(\;)^*$ is antitone, $(\Gamma z)^{**}\subseteq(\Gamma x)^*$. But then $\Gamma z\subseteq(\Gamma z)^{**}\subseteq(\Gamma x)^*$ from which we obtain $z\in(\Gamma x)^*$ and so $z\perp x$ follows.
\end{proof}

\begin{lemma}\label{ex falso frame lemma}\rm
The following are equivalent
    \begin{enumerate}
    \item $\perp$ is irreflexive
    \item $A\cap A^*=\emptyset$.
    \end{enumerate}
\end{lemma}
\begin{proof}
If $x\in A\cap A^*\neq\emptyset$, then it must be that $x\perp x$ and thus $\perp$ is not irreflexive. Conversely, assume $A\cap A^*=\emptyset$, for any $A\in\gpsi$, but suppose that $x\perp x$ for some $x\in X$. Then $x\in(\Gamma x)^*$ and thus $\Gamma x\cap(\Gamma x)^*\neq\emptyset$, contradiction.
\end{proof}

It remains to examine the case where $(\;)^*$ is an involution. The same map $\nu$ is now both a minimal and a dual minimal quasi-complement (see Remarks~\ref{dual minimal rem} and \ref{involution rem}).

Let $\mathfrak{F}=(X,\upv,Y,S_\vee)$ be a frame satisfying axioms (F0)--(F4) of Table \ref{quasiframe axioms table} and where $\bot=S'_\vee$ is symmetric. Define a relation $R_\wedge\subseteq X\times Y$ by $xR_\wedge y$ iff $y\leq\widehat{\nu}(x)$. The analogues of the frame axioms (F2)--(F3) are assumed for the derived relation $R_\wedge$, as additional frame axioms, listed as (I1)--(I3) in Table~\ref{additional axioms table}.

\begin{table}[t]
\caption{Additional frame axioms for the case of involution}
\label{additional axioms table}
\begin{enumerate}
  \item[(I1)] For any $y\in Y$, the section $R_\wedge y$ of the derived relation $R_\wedge\subseteq X\times Y$, defined by $xR_\wedge y$ iff $y\leq \widehat{\nu}(x)$, is a closed element of $\gpsi$.
  \item[(I2)] For any $x\in X$, the section $xR_\wedge$ of the  relation $R_\wedge$ is decreasing (a downset).
  \item[(I3)] Both sections of the Galois dual relation $R'_\wedge\subseteq Y\times Y$ of $R_\wedge$ are Galois sets.
  \item[(I4)] For all $x\in X$, $\widetilde{\nu}(\widehat{\nu}(x))=x$
\item[(I5)] For all $y\in Y$, $\widehat{\nu}(\widetilde{\nu}(y))=y$.
\end{enumerate}
\hrulefill
\end{table}

Designate the generating point of $R_\wedge y$, for each $y\in Y$, by $\widetilde{\nu}(y)$. Since $R_\wedge y=\Gamma\widetilde{\nu}(y)\in\gpsi$, we obtain that $R'_\wedge y=\{\widetilde{\nu}(y)\}\rperp\in\gphi$.

Let $\eta_R:\powerset(Y)\lra\powerset(X)$ be the sorted image operator generated by the (sorted) relation $R_\wedge$, defined by equation~\eqref{eta-R-defn}, repeated below 
\[
\eta_R(V)=\{x\in X\midsp\exists y\in Y(xR_\wedge y\wedge y\in V)\}=\bigcup_{y\in V}R_\wedge y
\]
and let $\overline{\eta}_R:\gphi\lra\gpsi$ be the closure of its restriction to Galois (co-stable) sets, defined by $\overline{\eta}_R(B)=\bigvee_{y\in B}R_\wedge y$, and set $\overline{\eta}_\wedge(A)=\overline{\eta}_R(A')=\bigvee_{A\upv y}R_\wedge y$.

As explained in Remarks~\ref{dual minimal rem} and \ref{involution rem}, the analogues of 
 Lemma~\ref{s-nu-gamma} and Proposition~\ref{eta-S-distrib-property} hold, by essentially the same arguments, and we obtain that $\overline{\eta}_R(\Gamma y)=R_\wedge y=\eta_R(\Gamma y)=\Gamma\widetilde{\nu}(y)$ and that $\overline{\eta}_R(\bigvee_{i\in I}B_i)=\bigvee_{i\in I}\overline{\eta}_R(B_i)$. Hence, $\overline{\eta}_\wedge(\bigcap_{i\in I}A_i)=\overline{\eta}_R((\bigcap_{i\in I}A_i)') =\overline{\eta}_R(\bigvee_{i\in I}A'_i)=\bigvee_{i\in I}\overline{\eta}_R(A'_i)=\bigvee_{i\in I}\overline{\eta}_\wedge(A_i)$. Summarizing,
\begin{align}
\overline{\eta}_\vee(\bigvee_{i\in I}A_i) &= \bigcap_{i\in I}\overline{\eta}_\vee(A_i)\\
\overline{\eta}_\wedge(\bigcap_{i\in I}A_i) &= \bigvee_{i\in I}\overline{\eta}_\wedge(A_i)
\end{align}
and we now aim at proving that $\overline{\eta}_\vee=\overline{\eta}_\wedge$. For this to hold, two additional axioms need to be in place, namely axioms (I4) and (I5) in Table~\ref{additional axioms table}.

\begin{lemma}
  \label{involution frame lemma}\rm
Let $\mathfrak{F}=(X,\upv,Y,S_\vee)$ be a frame satisfying axioms (F0)--(F4) of Table \ref{quasiframe axioms table}. Let also  the relation $R_\wedge\subseteq X\times Y$ be defined by $xR_\wedge y$ iff $y\leq\widehat{\nu}(x)$, as above, and assume that the additional axioms (I1)--(I5) in Table~\ref{additional axioms table} hold. If the incompatibility relation $\bot=S'_\vee$ is symmetric, then $\overline{\eta}_\vee=\overline{\eta}_\wedge$ and, consequently, the map $A^*=\overline{\eta}_\vee(A)$ is an involution on $\gpsi$.
\end{lemma}
\begin{proof}
We calculate that 
\begin{tabbing}
$\overline{\eta}_\wedge\overline{\eta}_\vee(A)    $\hskip2mm\= = \hskip2mm\= $\overline{\eta}_\wedge\overline{\eta}_\vee(\bigvee_{x\in A}\Gamma x)$ \hskip8mm\= By join-density of closed elements
\\
\>=\> $\overline{\eta}_\wedge(\bigcap_{x\in A}\overline{\eta}_\vee(\Gamma x))$ \>  $\overline{\eta}_\vee$ co-distributes over joins
\\
\>=\> $\overline{\eta}_\wedge(\bigcap_{x\in A}\rperp\{\widehat{\nu}(x)\})$  \>  $\overline{\eta}_\vee(\Gamma x)=(\overline{\eta}_S(\Gamma x))'=\rperp\{\widehat{\nu}(x)\}$, Lemma~\ref{s-nu-gamma}
\\
\>=\> $\bigvee_{x\in A}\overline{\eta}_\wedge(\rperp\{\widehat{\nu}(x)\})$   \> $\overline{\eta}_\wedge$ co-distributes over meets
\\
\>=\> $\bigvee_{x\in A}\overline{\eta}_R(\Gamma\widehat{\nu}(x))$ \> By definition $\overline{\eta}_\wedge(A)=\overline{\eta}_R(A')$\\
\>=\>$\bigvee_{x\in A}\Gamma(\widetilde{\nu}(\widehat{\nu}(x)))$ \> By $\overline{\eta}_R(\Gamma y)=\Gamma\widetilde{\nu}(y)$
\\
\>=\> $\bigvee_{x\in A}\Gamma x$\> By axiom (I4)
\\
\>=\> $A.$\> By join-density of closed elements
\\[3mm]
$\overline{\eta}_\vee\overline{\eta}_\wedge(A)$ \>=\> $\overline{\eta}_\vee\overline{\eta}_\wedge(\bigcap_{A\upv y}\rperp\{y\})$ \>   By meet-density of open elements
\\
\>=\> $\overline{\eta}_\vee(\bigvee_{A\upv y}\overline{\eta}_\wedge(\rperp\{y\}))$ \> $\overline{\eta}_\wedge$ co-distributes over meets
\\ 
\>=\> $\overline{\eta}_\vee(\bigvee_{A\upv y}\overline{\eta}_R(\Gamma y))$ \> By definition $\overline{\eta}_\wedge(A)=\overline{\eta}_R(A')$
\\
\>=\> $\overline{\eta}_\vee(\bigvee_{A\upv y}\Gamma\widetilde{\nu}(y))$ \> By $\overline{\eta}_R(\Gamma y)=\Gamma\widetilde{\nu}(y)$
\\
\>=\> $\bigcap_{A\upv y}\overline{\eta}_\vee(\Gamma\widetilde{\nu}(y))$\> $\overline{\eta}_\vee$ co-distributes over joins
\\
 \>=\> $\bigcap_{A\upv y}\rperp(\overline{\eta}_S(\Gamma\widetilde{\nu}(y)))$ \> By definition $\overline{\eta}_\vee(A)=(\overline{\eta}_S(A))'=\rperp(\overline{\eta}_S(A))$
 \\
\>=\> $\bigcap_{A\upv y}\rperp(\Gamma\widehat{\nu}(\widetilde{\nu}(y)))$\> By $\overline{\eta}_S(\Gamma x)=\Gamma\widehat{\nu}(x)$
\\
\>=\> $\bigcap_{A\upv y}\rperp\{\widehat{\nu}(\widetilde{\nu}(y))\}$\> By $\rperp(\Gamma y)=\rperp\{y\}$, Lemma~\ref{basic facts}
\\
 \>=\> $\bigcap_{A\upv y}\rperp\{y\}$\> By axiom (I5)
 \\
\>=\> $A$.\> By meet-density of open elements
\end{tabbing}
In particular, we then have $A\subseteq\overline{\eta}_\vee(C)$ iff $\overline{\eta}_\wedge\overline{\eta}_\vee(C)\subseteq\overline{\eta}_\wedge(A)$ iff $C\subseteq\overline{\eta}_\wedge(A)$, so that $\overline{\eta}_\vee,\overline{\eta}_\wedge$ form a Galois connection on $\gpsi$ (in fact they form a dual equivalence, as just proven above). However, we have already established, see equation~\eqref{eta-zeta galois}, that $\overline{\eta}_\vee$ is Galois connected with the map $\overline{\zeta}_\vee$ and thereby, by uniqueness of adjoints, we conclude that $\overline{\eta}_\wedge=\overline{\zeta}_\vee$. 

It then follows that if, in addition, the incompatibility relation $\bot=S'_\vee$ is symmetric, then using Lemma~\ref{galois frame lemma} we conclude that $\overline{\eta}_\wedge=\overline{\zeta}_\vee=\overline{\eta}_\vee$.
\end{proof}

\begin{coro}[Correspondence]\rm
\label{frame conditions for complex algebra coro}
Let $\mathfrak{F}=(X,\upv,Y,S_\vee)$ be a frame satisfying axioms (F0)--(F4) of Table \ref{quasiframe axioms table} and let $\mathfrak{F}^+=(\gpsi,\subseteq,\bigcap,\bigvee,\emptyset,X,(\;)^*)$ be  its full complex algebra of Galois stable sets.
\begin{enumerate}
  \item $\mathfrak{F}^+$ is a complete lattice with a minimal quasi-complementation operator $(\;)^*$ on stable sets.
  \item $\mathfrak{F}^+$ is a complete lattice with a quasi-complementation operator $(\;)^*$ which is a Galois connection on stable sets iff $\perp$ is symmetric.
  \item $\mathfrak{F}^+$ is a complete lattice with an involution $(\;)^*$ iff (a) $\perp$ is symmetric and (b) axioms (I1)--(I5) in Table~\ref{additional axioms table} hold in the frame.
  \item $\mathfrak{F}^+$ is a complete ortholattice iff (a) $\perp$ is symmetric and (b) irreflexive and (c)  axioms (I1)--(I5)  in Table~\ref{additional axioms table} hold in the frame.
  \item $\mathfrak{F}^+$ is a complete De Morgan algebra if (a) $\perp$ is symmetric, (b)   axioms (I1)--(I5) hold in the frame and (c) the sections of the Galois dual relation $R'_\leq$ of the upper bound relation $R_\leq\subseteq X\times (X\times X)$, defined by $uR_\leq xz$ iff $x\leq u$ and $z\leq u$,  are stable.
  \item $\mathfrak{F}^+$ is a complete Boolean algebra if the conditions (a)--(c) of the previous case hold in the frame and (d) $\perp$ is irreflexive.
\end{enumerate}
\end{coro}
\begin{proof}
  Immediate, given Lemmas~\ref{galois frame lemma}, \ref{ex falso frame lemma}, \ref{involution frame lemma} and Theorem \ref{upper bound rel prop}.
\end{proof}

\begin{coro}[Soundness]\label{soundness}\rm
Each of the cases in Corollary~\ref{frame conditions for complex algebra coro} implies that the logic of the respective variety is sound in its class of frames, as this class is determined by the respective conditions for each case.\telos
\end{coro}

\section{Choice-Free Representation of NLEs}
\label{rep section}
In our lattice representation and duality \cite{iulg,sdl} the strategy followed has been to derive the representation of a lattice $\mathbf{L}$ by combining the representations of its two semilattice reducts $\mathbf{L}_\wedge,\mathbf{L}_\vee$, by means of the representation of the trivial Galois connection (the identity map) $\imath:\mathbf{L}_\wedge\iso(\mathbf{L}_\vee)^\mathrm{op}$ as a binary relation, generalizing on Goldblatt's representation \cite{goldb} of ortholattices. We take the same approach here, except for switching from a Stone topology to a spectral one, hence we start by discussing the semilattice case.

\subsection{Semilattice Representation}
\label{section: sem rep}
Throughout the rest of this article principal filters of meet semilattices and lattices are typically designated with the notation $x_a$ ($=a{\uparrow}$), for an element $a$ of the (semi)lattice. Join semilattice and lattice principal ideals are similarly designated by $y_a=a{\downarrow}$. We typically use $x,z$ for filters, $y,v$ for ideals and $u,w$ for either case.

Let $\mathbf{M}=(M,\leq, \wedge, 1)$ be a meet semilattice with a unit (top) element $1$ and $X=\filt(\mathbf{M})$ its set of proper filters (we assume filters are nonempty, i.e. $1\in x$ for any $x\in X$). For each $a\in M$, let $X_a=\{x\in X\midsp a\in x\}$ and $\mathcal{B}=\{X_a\subseteq X\midsp a\in M\}$ and notice that $X_a\cap X_b=X_{a\wedge b}\in\mathcal{B}$, so that $\mathcal{B}$ itself is a meet semilattice with unit element $X=X_1$.

Let $\mathfrak{X}=(X,\mathcal{B})$ be the topological space with carrier set $X$ and topology $\Omega$ generated by taking $\mathcal{B}$ as a basis. We refer to $\mathfrak{X}$ as the {\em dual space} of $\mathbf{M}$.

\begin{lemma}\rm
\label{sobriety lemma}
Let $\mathbf{M}$ and $\mathfrak{X}$ be as above.
Given a filter $\mathcal{F}$ in the lattice $\Omega(\mathfrak{X})$ of open sets of $\mathfrak{X}$,
  define $\sub{x}{\mathcal{F}}=\{a\in M\midsp X_a\in\mathcal{F}\}$,  and observe that  $\sub{x}{\mathcal{F}}$ is a filter of $\mathbf{M}$, since $\mathcal{F}$ is a filter.
\begin{enumerate}
\item  For any basic open set $X_a\in\mathcal{B}$, $\sub{x}{\mathcal{F}}\in X_a$ iff $X_a\in\mathcal{F}$.
\item  For any open set $U$ of $\mathfrak{X}$, if $\sub{x}{\mathcal{F}}\in U$, then $U\in\mathcal{F}$.
\item  If $\mathcal{F}$ is a completely prime filter in the lattice $\Omega(\mathfrak{X})$ of open sets of $\mathfrak{X}$, then $\sub{x}{\mathcal{F}}\in U$ iff $U\in\mathcal{F}$.
\end{enumerate}
\end{lemma}
\begin{proof}
  Claim (1) is obvious, following from the definition of $\sub{x}{\mathcal{F}}$.

For (2), if $U$ is open, let $E\subseteq M$ be such that $U=\bigcup_{e\in E}X_e$. If $\sub{x}{\mathcal{F}}\in U$, then let $e\in E$ be an element such that $\sub{x}{\mathcal{F}}\in X_e$. By (1), $X_e\in\mathcal{F}$ and then since $X_e\subseteq U$ and $\mathcal{F}$ is a filter, $U\in\mathcal{F}$ follows.

For (3), it suffices to show, given (2) above, that if $\mathcal{F}$ is completely prime and $U\in\mathcal{F}$, then $\sub{x}{\mathcal{F}}\in U$. Now $U=\bigcup_{e\in E}X_e$ for some $E\subseteq M$ and then by complete primeness of $\mathcal{F}$ we get $X_e\in\mathcal{F}$, for some $e\in E$. It then follows by part (1) that $\sub{x}{\mathcal{F}}\in X_e\subseteq U$.
\end{proof}

Recall that for a point $x$ in a topological space, its open neighborhood $N^o(x)$ is defined as the set of open sets containing $x$. This gives rise to the definition of a partial order induced by the topology, the specialization order, defined by $x\sqleq u$ iff $N^o(x)\subseteq N^o(u)$.
Given any space $\mathfrak{X}$, a {\em filter} $F$ of $\mathfrak{X}$ is a non-empty upper set (with respect to the specialization order $\sqsubseteq$ on $\mathfrak{X}$) such that for any $x,z\in F$ a lower-bound  $u\in X$ of $\{x,z\}$ is in $F$. We let ${\tt KOF}(\mathfrak{X})$ designate the family of compact-open filters of $\mathfrak{X}$, following the notation of \cite{Moshier2014a}.

We draw the reader's attention to the fact that the proof of compactness in the next proposition considers an open cover by sets in the basis of the topology. This is in contrast with our proof of the same result in \cite{duality2}, where the Stone topology considered in \cite{duality2} is generated by a subbasis and, thereby, an appeal to the Alexander subbasis lemma was involved, thus making the proof of \cite{duality2} depend on assuming Zorn's lemma.

\begin{prop}\label{spectral prop}
  \rm
  Let $\mathbf{M}$ be a meet semilattice, $\mathfrak{X}=(X,\mathcal{B})$ its dual topological space (where $X=\filt(\mathbf{M})$ and $\mathcal{B}=\{X_a\midsp a\in M\}$ is a basis for the topology $\Omega$ on $X$).
\begin{enumerate}
 \item  The space $\mathfrak{X}$ is a spectral space.
 \item $\mathcal{B}=\{X_a\midsp a\in M\}={\tt KOF}(\mathfrak{X})$.
\end{enumerate}
\end{prop}
\begin{proof}
  Recall that a topological space is {\em spectral} if it is $T_0$,  coherent, compact and sober, which we prove in turn in order to establish part (1).

  For the $T_0$ property, if $x\neq z$ are distinct filters, without loss of generality we may assume that $a\in x$, but $a\not\in z$, for some semilattice element $a\in M$. Then the open set $X_a$ separates $x,z$, since $x\in X_a$ but $z\not\in X_a$.

  For  the coherence property we verify that the  basis $\mathcal{B}$ of the topology consists of compact-open sets and that it is closed under finite intersections. For the latter requirement,
  $\mathcal{B}$ is easily seen to be closed under finite intersections, since $X_a\cap X_b=X_{a\wedge b}$ and the intersection of the empty family of $X_a's$ is $X$ itself, which is identical to $X_1=\{x\in X\midsp 1\in x\}$. For the first requirement of coherence, the $X_a's$ are certainly open, by definition of the topology. For compactness, let $C\subseteq M$ and suppose that $\{X_e\midsp e\in C\}$ covers $X_a$, i.e. $X_a\subseteq\bigcup_{e\in C}X_e$.  Then the principal filter $x_a=a\!\uparrow{\in}\; X_a$ is in $X_e$, for some $e\in C$, hence $e\in x_a$, i.e. $a\leq e$. Then $e\in x$, for any $x\in X_a$ and this shows that $X_a\subseteq X_e=\{z\in X\midsp e\in z\}$. Hence $\{X_e\}$, for this $e$, is the needed finite subcover of $X_a$.

  Since $X=X_1=\{x\in X\midsp 1\in x\}$, compactness of $\mathfrak{X}$ follows from the previous argument.

  Sobriety of the space is equivalent to the requirement that every completely prime filter $\mathcal{F}$ in the lattice $\Omega(\mathfrak{X})$ of open sets of $\mathfrak{X}$ is generated by a single point $\sub{x}{\mathcal{F}}$, in other words that $\mathcal{F}=\{U\in\Omega(\mathfrak{X})\midsp \sub{x}{\mathcal{F}}\in U\}$. This was shown to hold in Lemma \ref{sobriety lemma}, part (3).

For part (2), left to right, $X_a$ is compact-open, by the proof of coherence for $\mathfrak{X}$ in part (1) of this proposition. Furthermore, $x\in X_a$ iff $a\in x$ iff $x_a\subseteq x$ iff $x\in\Gamma x_a$. Hence $X_a$ is a (principal) filter.

Conversely, let $F\subseteq X$ be a compact-open filter of $X$.  Being open, let $F=\bigcup_{a\in E\subseteq M}X_a$, so that by compactness, $F=X_{a_1}\cup\cdots\cup X_{a_n}$ for some $n$. By $a_i{\uparrow}=x_{a_i}\in X_{a_i}\subseteq F$, all the $x_{a_i}$'s, for $i=1,\ldots, n$, are in $F$, hence so is their meet (intersection), since $F$ is a filter. We let $u=\bigcap_{i=1}^nx_{a_i}$ and show that $F=\Gamma u=u{\uparrow}$. The right-to-left inclusion is obvious since $u\in F$, which is a filter, so $\Gamma u\subseteq F$. For the converse inclusion, let $z\in F$, so that $z\in X_{a_k}=\Gamma x_{a_k}$ for some $k\in\{1,\ldots,n\}$, hence $x_{a_k}\subseteq z$. By definition of $u$ we obtain $u\subseteq x_{a_k}\subseteq z$ and then $z\in\Gamma u$.  Hence $F\subseteq\Gamma u$ follows, too. Hence $F=\Gamma u=\bigcup_{i=1}^nX_{a_i}$ and thus $u\in X_{a_i}=\Gamma x_{a_i}$ for some $i$. But then $x_{a_i}\subseteq u=\bigcap_{i=1}^nx_{a_i}\subseteq x_{a_i}$ so that $u=x_{a_i}$, for this $i$, and so $F=\Gamma x_{a_i}=X_{a_i}\in\mathcal{B}$.
\end{proof}

It should be pointed out that, except for the phrasing, notation and detail, the arguments in the proofs of Lemma \ref{sobriety lemma} and Proposition \ref{spectral prop} are the same as these involved in showing that the space of proper filters of a Boolean algebra is spectral \cite[Proposition 3.12]{choice-free-BA}, or that the space of proper filters of an ortholattice is spectral \cite[Proposition 3.4.1]{choice-free-Ortho}. It is really only the semilattice-structure that is relevant in the argument, which is one of the reasons that we included a proof of Proposition \ref{spectral prop}, the other reason relating to the observation made in \cite{sdl,duality2} that a lattice can be always regarded as a diagram of dually isomorphic meet semilattices. For the case of Boolean algebras and ortholattices, this duality may be taken to be Boolean complementation, or ortho-complementation, respectively. For general bounded lattices the semilattice duality can be taken to be the identity map $\imath:\mathbf{L}_\wedge\leftrightarrows(\mathbf{L}_\vee)^\partial$, where $\mathbf{L}_\vee=(\mathbf{L}_\wedge)^\partial$, as in \cite{sdl,duality2}, and as we explain in more detail in the sequel.

Note that the topology $\Omega$ on $X$ is the Scott topology, with respect to the specialization order $\sqleq$ on $X$ \cite[chapter II.1]{stone-spaces},
which is inclusion of filters (of $\mathbf{M}$).
To see that specialization coincides with filter inclusion note that if $x\sqleq z$, i.e. $N^o(x)\subseteq N^o(z)$ and $a\in x$, then $x\in X_a\in N^o(x)\subseteq N^o(z)$, hence also $z\in X_a$, i.e. $a\in z$ and so $x\subseteq z$. Conversely, if $x\subseteq z$ and $U$ is an open neighborhood of $X$, let $X_a$ be a basic open such that $x\in X_a\subseteq U$. From $x\in X_a$ we get $a\in x\subseteq z$, so also $z\in X_a\subseteq U$, hence $U\in N^o(z)$, i.e. $N^o(x)\subseteq N^o(z)$ which by definition means that $x\sqleq z$.

It follows from Proposition \ref{spectral prop} that the space $\mathfrak{X}$ is an HMS space (named so in \cite{Moshier2014a},   in honour of Hofmann, Mislove and Stralka), defined by a set of equivalent conditions in \cite[Theorem 2.5]{Moshier2014a}).

The following representation result is an immediate consequence of Proposition \ref{spectral prop}.

\begin{coro}
\label{kof coro}
  \rm
  Given a meet semilattice $\mathbf{M}$, the map $a\mapsto X_a$ is a semilattice isomorphism $\mathbf{M}\iso{\tt KOF}(\filt(\mathbf{M}))$.\telos
\end{coro}

\subsection{The Canonical Dual Space of a Lattice}
\label{lat rep section}
If $\mathbf{N}$ is a join semilattice with unit (bottom) element 0, then $\mathbf{N}^\partial$ (the opposite semilattice, order reversed) is a meet semilattice with unit (top) and $\filt(\mathbf{N}^\partial)=Y=\idl(\mathbf{N})$, the set of proper ideals of $\mathbf{N}$. The topology generated by the basis of sets $Y^a=\{y\in Y\midsp a\in y\}$, for $a\in N$, is a spectral topology by Proposition \ref{spectral prop}, observing that $Y^a\cap Y^b=Y^{a\vee b}$, which ensures that the basis $\mathcal{C}=\{Y^a\midsp a\in N\}$ is closed under finite intersections (with the empty intersection being $Y_0=Y$ itself).

For the lattice case, just as orthonegation is represented in \cite{goldb} by the binary relation ${\perp}\subseteq X\times X$ defined by $x\perp z$ iff $\exists a(a\in x\wedge a^\perp\in z)$, the identity trivial duality $\imath:\mathbf{L}\iso(\mathbf{L}^\partial)^\partial$ is similarly represented \cite{ sdl,duality2} by the sorted binary relation ${\upv}\subseteq X\times Y$ defined by $x\upv y$ iff $\exists a(a\in x\wedge\imath(a)\in y)$ iff $x\cap y\neq\emptyset$.

Note that the quasi-seriality condition set by axiom (F0) holds for the complement of the canonical Galois relation $\upv$.

As in \cite{ sdl}, we represent lattices and normal lattice expansions, more generally, in topologized sorted frames (polarities) $\mathfrak{F}=(X,\upv,Y,(R_\sigma)_{\sigma\in\tau})$, where for each normal lattice operator $f$ of distribution type $\sigma=(i_1,\ldots,i_n;i_{n+1})$ the frame is equipped with a sorted relation $R_\sigma$ of sort $\sigma=(i_{n+1};i_1\cdots i_n)$, i.e. $R_\sigma\subseteq Z_{i_{n+1}}\times\prod_{j=1}^{n}Z_{i_j}$ and where $Z_{i_j}=X$ when $i_j=1$ and $Z_{i_j}=Y$ when $i_j=\partial$.

\begin{prop}\rm
\label{compact-open stable prop}
  For a bounded lattice $\mathbf{L}$, the bases $\mathcal{B}=\{X_a\midsp a\in L\}$ and $\mathcal{C}=\{Y^a\midsp a\in L\}$ of the spaces $\mathfrak{X}=(X,\mathcal{B}), \mathfrak{Y}=(Y,\mathcal{C})$, where $X=\filt(\mathbf{L})$ and $Y=\idl(\mathbf{L})$, are the families of compact-open Galois stable and co-stable, respectively, sets. Furthermore, $\mathcal{B}$ and $\mathcal{C}$ are dually isomorphic bounded lattices, with $\mathcal{B}$ a sublattice of $\gpsi$ and $\mathcal{C}$ a sublattice of $\gphi$.
\end{prop}
\begin{proof}
The two cases for  $\mathcal{B}$ and $\mathcal{C}$ are similar. The proof follows from \cite[Lemma~2.7]{sdl}. That lemma is phrased in terms of clopen sets in a Hausdorff space, which are then compact-open and compactness and (co)stability are the only properties needed and used in the argument. Other than that, the claim in \cite[Lemma~2.7]{sdl} is phrased in terms of an arbitrary duality $\ell:\mathbf{S}\iso\mathbf{K}^\partial:r$ between meet semilattices. The case for lattices follows by specializing the argument to the trivial duality $\ell=r=\imath:\mathbf{L}_\wedge\leftrightarrows(\mathbf{L}_\vee)^\partial$, which we do below.

First, stability of the sets $X_a, Y^a$ follows from Lemma \ref{basic facts}, observing that $X_a=\Gamma x_a$ and $Y^a=\Gamma y_a$. It remains to show that every stable compact-open subset $A$ of $X$ is of the form $X_a$, for some lattice element $a\in L$. 

Assume $A=A''$ is compact-open and let $x\not\in A= A''$. Let $y\in A'$ be an ideal such that $x\not\upv y$, i.e. $x\cap y=\emptyset$. By $y\in A'$ we have $A\upv y$, i.e. for any $z\in A$ we have $z\upv y$, i.e. $a_z\in z\cap y$, for some lattice element $a_z$. Thus $A\subseteq \bigcup_{z\in A}X_{a_z}$ and, by compactness, it follows that $A\subseteq X_{a_{z_1}}\cup\cdots\cup X_{a_{z_n}}$, for some $n$. Letting $a_x=a_{z_1}\vee\cdots\vee a_{z_n}$ it follows that for all $i=1,\ldots,n$, $X_{a_{z_i}}\subseteq X_{a_x}$, hence $A\subseteq X_{a_x}$. Notice that $a_x\not\in x$, since $a_x\in y$. Hence $x\not\in X_{a_x}$. This shows that the complement of $A$ is contained in the complement of $X_{a_x}$, $-A\subseteq -X_{a_x}$, and given we also obtained $A\subseteq X_{a_x}$ it follows that $A=X_{a_x}$.

That  both $\mathcal{B,C}$ are (dually isomorphic) lattices follows from the fact that $(X_a)'=Y^a$ and $(Y^a)'=X_a$. Joins in $\mathcal{B}, \mathcal{C}$ are defined by taking closures of unions: $A\vee C=(A\cup C)''$, as in $\gpsi$ and $\gphi$.
\end{proof}

Let ${\tt KO}\mathcal{G}(\mathfrak{X}),{\tt KO}\mathcal{G}(\mathfrak{Y})$ be the families of Galois compact-open subsets in $\mathfrak{X}$ and in $\mathfrak{Y}$, respectively. By Propositions \ref{spectral prop} and \ref{compact-open stable prop}, ${\tt KOF}(\mathfrak{X})={\tt KO}\mathcal{G}(\mathfrak{X})$ and ${\tt KOF}(\mathfrak{Y})={\tt KO}\mathcal{G}(\mathfrak{Y})$.

The following choice-free lattice representation theorem is an immediate consequence of our results so far in this article.

\begin{thm}[Choice-free Lattice Representation]\label{choice-free lat rep}
  \rm
  Let $\mathbf{L}=(L,\leq,\wedge,\vee,0,1)$ be a bounded lattice and $(X,\upv,Y)$ its dual filter-ideal frame ($X=\filt(\mathbf{L}),Y=\idl(\mathbf{L})$), with ${\upv}\subseteq X\times Y$ defined by $x\upv y$ iff $x\cap y\neq\emptyset$. Let $\mathfrak{X}=(X,\mathcal{B})$ and $\mathfrak{Y}=(Y,\mathcal{C})$ be the spectral spaces generated by the bases $\mathcal{B}=\{X_a\midsp a\in L\}$ and $\mathcal{C}=\{Y^a\midsp a\in L\}$, respectively.

  Then the map $a\mapsto X_a$ is a lattice isomorphism $\mathbf{L}\iso{\tt KO}\mathcal{G}(\mathfrak{X})$ and the map $a\mapsto Y^a$ is a dual isomorphism $\mathbf{L}^\partial\iso{\tt KO}\mathcal{G}(\mathfrak{Y})$.\telos
\end{thm}

The representation of (semi)lattices detailed above is essentially the same as that in \cite{sdl}, by this author and Dunn, the difference lying in the choice of the topology imposed on the filter space.
The lattice $\gpsi$ of Galois stable sets is a canonical extension of the lattice, see \cite[Proposition~2.6]{mai-harding}, which is unique up to an isomorphism that commutes with the lattice embedding, by \cite[Proposition~2.7]{mai-harding}.

Moshier and Jipsen \cite{Moshier2014a} provide a topological construction of the canonical extension of a lattice.
Proposition \ref{compact-open stable prop}, together with uniqueness of canonical extensions up to isomorphism, entails  that the filter space of a lattice is what Moshier and Jipsen call a ${\tt BL}$-space, defined by a number of equivalent conditions in \cite[Theorem~3.2]{Moshier2014a}. It can be easily verified that the canonical extension ${\tt FSat}(X)$ defined in \cite{Moshier2014a} is literally identical to $\gpsi$. We substantiate this claim below. Recall first that ${\tt OF}(X)$ designates in \cite{Moshier2014a} the family of open filters of $X=\filt(\mathbf{L})$ and that the closure operator {\tt fsat} is defined by ${\tt fsat}(U)=\bigcap\{F\in{\tt OF}(X)\midsp U\subseteq F\}$.

\begin{prop}\label{fsat prop}
  \rm
  The following hold:
\begin{enumerate}
\item  A Galois stable set $A=A''$ is a filter of $X=\filt(\mathbf{L})$ and, similarly, a Galois co-stable set $B=B''$ is a filter of $Y=\idl(\mathbf{L})$.
\item Every open filter $F\in{\tt OF}(X)$ is of the form ${}\rperp\{y\}$ for an ideal $y\in Y=\idl(\mathbf{L})$, which is unique by separation of the frame (see Lemma \ref{basic facts}).
\item For any subset $U\subseteq X$, ${\tt fsat}(U)=U''$. Therefore, ${\tt FSat}(X)=\gpsi$.
\end{enumerate}
\end{prop}
\begin{proof}
  For part (1), the proof is straightforward and we only discuss the case of stable sets. First, Galois sets are upsets, by Lemma \ref{basic facts}. Now let $x,z\in A=A''$ and suppose that $x\cap z\not\in A''$. Then there exists an ideal $y\in A'$ such that $(x\cap z)\not\upv y$, i.e. $(x\cap z)\cap y=\emptyset$. Since $x,z\in A''$, we have $x\cap y\neq\emptyset$ and $z\cap y\neq\emptyset$, hence there exist lattice elements $a,b$ such that $a\in x\cap y$ and $b\in z\cap y$. This implies that both $a,b\in y$, hence $a\vee b\in y$. But $x,z$ are filters, hence $a\vee b\in x\cap z$, which contradicts the assumption that $x\cap z\not\in A''$.

  For part (2), for any $y\in Y$, ${}\rperp\{y\}$ is Galois stable, hence a filter of $X$, by part (1). To see that it is an open set, let $x\in{}\rperp\{y\}$, i.e. $x\upv y$, so that $a\in x\cap y\neq\emptyset$. Thus $x\in X_a$. Since $a\in y$, $X_a\upv y$, so that we obtain $x\in X_a\subseteq{}\rperp\{y\}$. Thus ${}\rperp\{y\}\in{\tt OF}(X)$, for any $y\in Y$.

  Conversely, let $F\in{\tt OF}(X)$ and let $E\subseteq L$ be such that $F=\bigcup_{a\in E}X_a$. Let $y$ be the ideal generated by $E$. Thus $e\in y$ iff there exist $e_1,\ldots,e_n\in E$, for some $n$, such that $e\leq e_1\vee\cdots\vee e_n$. We show that $F={}\rperp\{y\}$, for this $y$.

  If $x\in F=\bigcup_{a\in E}X_a$, then $x\in X_a$, for some $a\in E$. By definition of $y$, we get $a\in y$, so $x\upv y$, i.e. $x\in{}\rperp\{y\}$. Hence $F\subseteq{}\rperp\{y\}$. Conversely, let $x\upv y$ so that $e\in x\cap y$. Then $e\leq e_1\vee\cdots\vee e_n$, where $\{e_1,\ldots,e_n\}\subseteq E$. It follows that $x_{e_1}\cap\cdots\cap x_{e_n}=x_{e_1\vee\cdots\vee e_n}\subseteq x_e\subseteq x$. Since $e_i\in E$, we have $X_{e_i}\subseteq F$. Because $x_{e_i}\in X_{e_i}$, all principal filters $x_{e_1},\ldots,x_{e_n}\in F$. Since $F$ is a filter, their intersection is in $F$ and then also $x\in F$. Hence ${}\rperp\{y\}\subseteq F$.

  For part (3), by Lemma \ref{basic facts} the set of open elements ${}\rperp\{y\}$ is meet-dense in $\gpsi$. By part (2) above, ${\tt OF}(X)=\{{}\rperp\{y\}\midsp y\in Y\}$. Using also the definition of $F$-saturation in \cite{Moshier2014a} we obtain
  \[
  {\tt fsat}(U)=\bigcap\{F\in{\tt OF}(X)\midsp U\subseteq F\}=\bigcap\{{}\rperp\{y\}\midsp U\upv y\}=U''
  \]
  and so ${\tt FSat}(X)=\{A\subseteq X\midsp A={\tt fsat}(A)\}=\{A\subseteq X\midsp A=A''\}=\gpsi$.
\end{proof}

\subsection{Representing Normal Lattice Operators}
\label{normal ops rep section}
Let $\mathbf{L}=(L,\leq,\wedge,\vee,0,1,f)$ be a bounded lattice with a normal operator $f$. Then $f$ extends to a completely normal operator $F$, of the same distribution type as $f$, on the canonical extension $\mathcal{G}(\filt(\mathbf{L}))$ of $\mathbf{L}$.

For the proof, we refer the reader to \cite[Sections~3.1,4.1,4.2]{duality2}. The representation of the operator is the same as that given in \cite{sdl-exp}, the difference lying in the axiomatization of the dual frame of the lattice expansion, in particular on the axioms for the relations corresponding to the lattice operator. Particular instances of the representation were also given in \cite{pnsds} and, recently, in the choice-free duality for implicative lattices and Heyting algebras \cite{choiceFreeHA}.

To keep this article as self-contained as possible, we sketch the representation steps, drawing on \cite[Section~4.1]{duality2}.

The underlying polarity $\mathfrak{F}=(\filt(\mathcal{L}),\upv,\idl(\mathcal{L}))$ of the canonical frame consists of the sets $X=\filt(\mathcal{L})$ of proper filters and $Y=\idl(\mathcal{L})$ of proper ideals of the lattice and the relation ${\upv}\subseteq\filt(\mathcal{L})\times\idl(\mathcal{L})$, defined by $x\upv y$ iff $x\cap y\neq\emptyset$, while the representation map $\zeta_1$ sends a lattice element $a\in L$ to the set of proper filters that contain it, 
\[
\zeta_1(a)=X_a=\{x\in X\midsp a\in x\}=\{x\in X\midsp x_a\subseteq x\}=\Gamma x_a.
\]
Similarly, a co-represenation map $\zeta_\partial$ is defined by 
\[
\zeta_\partial(a)=\{y\in Y\midsp a\in y\}=\{y\in Y\midsp y_a\subseteq y\}=\Gamma y_a=Y^a.
\]
For each normal lattice operator of distribution type  $\delta=(i_1,\ldots,i_n;i_{n+1})$ a relation is defined, such that  $\sigma=(i_{n+1};i_1\cdots i_n)$ is the sort type of the relation. Without loss of generality, we may consider just two normal operators $f$, of output type $1$, and $h$, of output type $\partial$. We define two corresponding relations $R,S$ of respective sort types $\sigma(R)=(1;i_1\cdots i_n)$ and $\sigma(S)=(\partial;t_1\cdots t_n)$, where for each $j$, $i_j$ and $t_j$ are in $\{1,\partial\}$. In other words
$R\subseteq X\times\prod_{j=1}^{j=n}Z_{i_j}$, while \mbox{$S\subseteq Y\times \prod_{j=1}^{j=n}Z_{t_j}$.}

To define the relations, we use the point operators introduced in \cite{dloa} (see also \cite{sdl-exp}). In the generic case we examine, we need to define two sorted operators
\[
\widehat{f}:\prod_{j=1}^{j=n}Z_{i_j}\lra Z_1\hskip0.8cm
\widehat{h}: \prod_{j=1}^{j=n}Z_{t_j}\lra Z_\partial\hskip0.8cm(\mathrm{recall\ that\ }Z_1=X, Z_\partial=Y).
\]
Assuming for the moment that the point operators have been defined, the canonical relations $R,S$ are defined by \eqref{canonical relations defn}, where $\vec{u}=u_1\cdots u_n$ and similarly for $\vec{v}$,
\begin{align}\begin{split}
xR\vec{u} &\;\; \mbox{ iff }\;\; \widehat{f}(\vec{u})\subseteq x\;\; (\mbox{for }\; x\in X=\filt(\mathbf{L})\;\mbox{ and }\;\vec{u}\in \prod_{j=1}^{j=n}Z_{i_j}),\\
yS\vec{v} &\;\;\mbox{ iff }\;\; \widehat{h}(\vec{v})\subseteq y\;\; (\mbox{for }\; y\in Y=\idl(\mathbf{L})\;\mbox{ and }\;\vec{v}\in \prod_{j=1}^{j=n}Z_{t_j}).\label{canonical relations defn}
\end{split}
\end{align}
Returning to the point operators and letting $x_e,y_e$ be the principal filter and principal ideal, respectively, generated by a lattice element $e$, these are uniformly defined as follows, for $\vec{u}\in \prod_{j=1}^{j=n}Z_{i_j}$ and $\vec{v}\in \prod_{j=1}^{j=n}Z_{t_j}$, where $\vec{a}\in\vec{u}$ abbreviates the conjunction of membership statements $\bigwedge_{i=1}^n(a_i\in u_i)$,
\begin{equation}
  \widehat{f}(\vec{u}) \;=\; \bigvee\{x_{f(\vec{a})}\midsp\vec{a}\in\vec{u}\}\hskip1.5cm
  \widehat{h}(\vec{v}) \;=\; \bigvee\{y_{h(\vec{a})}\midsp\vec{a}\in\vec{v}\}.\label{canonical point operators defn}
\end{equation}
In other words, $\widehat{f}(\vec{u})$ is the filter generated by the set $\{f(\vec{a})\midsp \vec{a}\in\vec{u}\}$. Similarly $\widehat{h}(\vec{v})$ is the ideal generated by the set $\{h(\vec{a})\midsp \vec{a}\in\vec{v}\}$.

Though differently defined, the Moshier-Jipsen representation of quasioperators \cite{Moshier2014b} as strongly continuous and finite meet preserving point operators in the dual space is identical to the representation maps $\widehat{f},\widehat{h}$ defined above, as we detailed in \cite[Remark~4.2, Remark~4.8]{duality2}.

In the next section we instantiate definitions to the case of  quasi complemented lattices and we draw on \cite{duality2} to establish properties of the canonical frame.

\section{Representing Quasi-Complemented Lattices}
\label{rep quasi comp section}
In this section we extend the lattice representation of section \ref{lat rep section} to the case of a lattice with an additional quasi-complementation operator, assuming the axiomatization of at least the minimal system of Figure \ref{quasi-complements figure} and specializing the constructions of Section \ref{normal ops rep section}.

The canonical dual frame is the structure $(X,\upv,Y,S_\vee)$, where $X=\filt(\mathbf{L})$, $Y=\idl(\mathbf{L})$, ${\upv}\subseteq X\times Y$ is defined by $x\upv y$ iff $x\cap y\neq\emptyset$ and $S_\vee\subseteq Y\times X$ is the canonical relation defined using the point operator $\widehat{\nu}:X\lra Y$, according to \cite[Equation~(4.1), Equation~(4.2)]{duality2}. Instantiating to the case at hand, the definitions are given by equation \eqref{neg point op and rel}
\begin{equation}\label{neg point op and rel}
  \widehat{\nu}(x)=\bigvee\{y_{\nu a}\midsp a\in x\}\hskip1cm yS_\vee x\mbox{ iff }\widehat{\nu}(x)\subseteq y\hskip1cm (x\in X, y\in Y).
\end{equation}
where recall that $y_e$ designates the principal ideal generated by the lattice element $e$.

It follows from the above definitions that the canonical relation $S_\vee\subseteq Y\times X$ is equivalently defined by  condition \eqref{canonical-Snu}
\begin{equation}\label{canonical-Snu}
yS_\vee x\;\mbox{ iff }\; \forall a\in L(a\in x\lra\nu a\in y).
\end{equation}
Similarly we obtain from definitions that its Galois dual relation $S'_\vee\subseteq X\times X$ is defined by
\begin{equation}\label{galois dual defn}
zS'_\vee x\mbox{ iff }\forall y\in Y(yS_\vee x\lra z\upv y)\mbox{ iff }z\upv\widehat{\nu}(x)\mbox{ iff }\exists a\in L(a\in x\mbox{ and }\nu a\in z).
\end{equation}
Proof details are straightforward and they are instances of the proofs for the general case of arbitrary normal lattice operators in \cite[Lemma~4.4, Lemma~4.5]{duality2}).

Recall that $S_\vee x=\Gamma(\widehat{\nu}(x))$ so that $S'_\vee x=\{\widehat{\nu}(x)\}'={}\rperp\{\widehat{\nu}(x)\}$.

A sorted image operator $\eta_S:\powerset(X)\lra\powerset(Y)$ is generated by the relation $S_\vee$, defined  by equation \eqref{sorted image op}. The closure of the restriction of $\eta_S$ to stable sets is designated by $\overline{\eta}_S$, defined by equation \eqref{closure of sorted def}.

\begin{prop}\rm
\label{canonical section stability Snu}
In the canonical lattice frame all axioms of Table \ref{quasiframe axioms table} hold. In particular, both sections of the Galois dual relation $S'_\vee$ of the canonical relation $S_\vee$ are Galois sets.
\end{prop}
\begin{proof}
The proof for axioms (F1)--(F3) is given in \cite[Lemma~4.3]{duality2} for the general case of arbitrary normal lattice operators. For axiom (F4),
the claim was first stated as \cite[Lemma~25]{kata2z} and a proof of one of the subcases was detailed, the other one being sufficiently similar. The omitted proof of the other subcase was provided in \cite[Lemma~4.6]{duality2}. Axiom (F0) obviously holds in the canonical frame since every proper filter $x$ does not intersect the principal ideal $y_a$, for any $a\not\in x$. Similarly for ideals.
\end{proof}

By \cite[Theorem~3.12]{duality2}, given also Proposition \ref{canonical section stability Snu}, $\overline{\eta}_S$ distributes over arbitrary joins of stable sets in $\gpsi$, returning a join of co-stable sets in $\gphi$.

The canonical extension $\overline{\eta}_\vee:\gpsi\lra\gpsi$ of the normal lattice operator $\nu$ is obtained by composing $\overline{\eta}_S$ with the Galois connection, setting $\overline{\eta}_\vee(A)=(\overline{\eta}_S(A))'\in\gpsi$, hence $\overline{\eta}_\vee(A)=\bigcap_{x\in A}S'_\vee x$, where $S'_\vee$ is the Galois dual relation  of $S_\vee$.  Given \eqref{closure of sorted def}, we obtain $\overline{\eta}_\vee(A)=\bigcap_{x\in A}{}\rperp\{\widehat{\nu}(x)\}$.

In the terminology of \cite{mai-harding}  $\overline{\eta}_\vee=\nu^\pi$  is the $\pi$-extension of the lattice operator $\nu$, defined by
\begin{equation}\label{pi of nu}
\nu^\pi(A)=\bigcap_{x\in A}{}\rperp\{\widehat{\nu}(x)\}=\bigcap_{x\in A}S'_\vee x=\overline{\eta}_\vee(A)={}\rperp(\overline{\eta}_S(A))={}\rperp(\eta_S(A)).
\end{equation}
 That $\overline{\eta}_\vee$  is indeed the $\pi$-extension $\nu^\pi$ of the quasi-complementation operator $\nu$ follows from the general proof in  \cite[Section~3.3]{kata2z} that
the operator defined by suitable composition with the Galois connection (to obtain a single-sorted operator on $\gpsi$) of the closure $\overline{\alpha}_R$ of the restriction to Galois sets of the image operator $\alpha_R$ generated by a canonical frame relation $R$ as in equation~\eqref{canonical relations defn}, is indeed the canonical extensions ($\sigma$, or $\pi$-extension, according to its output type) of the respective lattice operator.

As in Section~\ref{section: frames for quasi comp lats}, Definition~\ref{image-op-and-bot}, we define the incompatibility relation ${\perp}\subseteq X\times X$ to be the Galois dual relation  $S'_\vee$ of the canonical frame relation $S_\vee$ and we let $A^*=\{x\midsp x\perp A\}=\{x\in X\midsp\forall z(z\in A\lra x\perp z)\}$. By Lemma~\ref{ast-is-eta-nu}, $A^*=\nu^\pi(A)=\overline{\eta}_\vee(A)$.

\begin{lemma}\label{restriction to clopens}\rm
If $\mathbf{L}$ is a quasi-complemented lattice, then the lattice isomorphism $\mathbf{L}\iso{\tt KO}\mathcal{G}(\filt(\mathbf{L}))$ is in fact an isomorphism of quasi-complemented lattices.
\end{lemma}
\begin{proof}
It suffices to show that the representation function, mapping a lattice element $a$ to the set $X_a=\{x\in\filt(\mathbf{L})\midsp a\in x\}=\Gamma x_a$ is a homomorphism of quasi-complemented lattices, i.e. that $(X_a)^*=X_{\nu a}$.
\begin{tabbing}
\hskip4mm\= $(X_a)^*$\hskip4mm\==\hskip2mm\= $(\Gamma x_a)^*=\overline{\eta}_\vee(\Gamma x_a)$\hskip2.5cm\= by Lemma~\ref{ast-is-eta-nu}\\
\>\>=\> $(\overline{\eta}_S(\Gamma x_a))'$\> by definition of $\overline{\eta}_\vee$\\
\>\>=\> $S'x_a$ \> by Lemma~\ref{s-nu-gamma}\\
\>\>=\> ${}\rperp\{\widehat{\nu}(x_a)\}$\> by the same Lemma.
\end{tabbing}
It follows from the definition of $\widehat{\nu}$ in equation \eqref{neg point op and rel} that for a principal filter $x_a$, the ideal $\widehat{\nu}(x_a)=\bigvee\{y_{\nu e}\midsp a\leq e\}$ is the principal ideal $y_{\nu a}$. Furthermore, for any principal ideal $y_c$ and filter $x$, $x\upv y_c$ iff $c\in x$. Thereby
\begin{tabbing}
\hskip4mm\= $(X_a)^*$\hskip4mm\==\hskip2mm\=${}\rperp\{\widehat{\nu}(x_a)\}$\\
\>\>=\> ${}\rperp\{y_{\nu a}\}=\{x\in X\midsp \nu a\in x\}$\\
\>\>=\> $X_{\nu a}$
\end{tabbing}
which proves the claim of the lemma.
\end{proof}

\begin{prop}\label{canonical extension prop}\rm
Let $\mathbf{L}$ be a quasi-complemented lattice and $\mathbf{L}_+=\mathfrak{F}$ be its canonical frame. Then the full complex algebra $\mathfrak{F}^+=(\gpsi,\subseteq,\bigcap,\bigvee,\emptyset,X,\nu^\pi)$, where we define $\nu^\pi(A)={}\rperp\eta_S(A)$, is a canonical extension of $\mathbf{L}$ with embedding map $h(a)=X_a=\{x\in X\midsp a\in x\}$.
\end{prop}
\begin{proof}
The proof that the lattice $\gpsi$ is a canonical extension of $\mathbf{L}$ was given in  \cite[Proposition~2.6]{mai-harding}. That $\nu^\pi$ is indeed the $\pi$-extension of the quasi complementation operator $\nu$ was pointed out already and the proof is an instance of the corresponding argument in \cite[Section~3.3]{kata2z}.
\end{proof}

It remains to finally show that for each of the quasi-complemented lattice varieties of Figure~\ref{quasi-complements figure},  ${\tt KO}\mathcal{G}(\filt(\mathbf{L}))$ is not merely a sublattice, but a subalgebra of the full complex algebra $\mathfrak{F}^+$ of the canonical frame $\mathfrak{F}=\mathbf{L}_+$ of $\mathbf{L}$ which, in addition, is a canonical extension of $\mathbf{L}$.

\begin{thm}\rm
  \label{fig 1 theorem}
  Let $\mathbf{L}=(L,\leq,\wedge,\vee,0,1,\nu)$ be a bounded lattice with an antitone map $\nu$. If $\mathbf{L}$ belongs to one of the lattice varieties of Figure \ref{quasi-complements figure}, then so does the full complex algebra $(\gpsi,\subseteq,\bigcap,\bigvee,\emptyset,X,\nu^\pi)$ of its canonical frame. In other words, each of the varieties of Figure \ref{quasi-complements figure} is closed under canonical extensions.
\end{thm}
\begin{proof}
We treat the cases for $\nu$ in turn.\\[1mm]
{\em Case 1: Minimal Quasi-complement.}

By the results of sections \ref{section: frames for quasi comp lats} and  \ref{normal ops rep section},  $\nu^\pi=(\;)^*:\gpsi\lra\gpsi$ co-distributes over arbitrary joins in $\gpsi$, returning a meet, i.e. $\left(\bigvee_{i\in I}A_i\right)^*=\bigcap_{i\in I}A_i^*$. Hence the full complex algebra $(\gpsi,\subseteq,\bigcap,\bigvee,\emptyset,X,(\;)^*)$ of the canonical frame
is a complete lattice with a minimal quasi-complementation operator, given that we also have $\emptyset^*=(\overline{\eta}_S(\emptyset))'=\emptyset'=X$, using normality of the classical, though sorted, image operator $\eta_S$.
\\[1mm]
{\em Case 2: Galois connection.}

If $\nu$ satisfies the Galois condition $a\leq\nu\nu a$, then it is immediate that the canonical relation $\perp=S'_\vee$, defined by equation \eqref{galois dual defn}, is symmetric. By Lemma \ref{galois frame lemma}, this implies that $A\subseteq A^{**}$, for any $A\in\gpsi$.
\\[1mm]
{\em Case 3: Involution.}

When $\nu$ is an involution, it has both distribution types $(1;\partial)$ and $(\partial;1)$, hence there are two different ways to obtain a representation, following the framework of \cite{duality2}, one being the $\sigma$-extension $\nu^\sigma$ and the other being the $\pi$-extension $\nu^\pi$ of $\nu$. We show that the $\sigma$-extension is the map $\overline{\eta}_\wedge$ generated by the derived relation $R_\wedge$ that we defined for the proof of Lemma~\ref{involution frame lemma} and that all additional frame axioms (I1)--(I5) of Table~\ref{additional axioms table} hold in the canonical frame. Figure~\ref{sigma-pi-canonical} displays, side by side, the relational construction of the $\sigma$ and $\pi$ extensions of the lattice involution $\nu$.

\begin{figure}[t]
\caption{Relational Construction of the $\sigma$ and $\pi$-extensions}
\label{sigma-pi-canonical}
\begin{tabbing}
$\pi$-extension, $\nu^\pi$\hskip3.5cm\= $\sigma$-extension, $\nu^\sigma$\\[2mm]
$\widehat{\nu}(x)$\hskip5mm\==\hskip2mm\= $\bigvee_{a\in x}y_{\nu a}$ ($\widehat{\nu}:X\lra Y$) \hskip7mm\= 
$\widetilde{\nu}(y)$\hskip5mm\= =\hskip2mm\= $\bigvee_{a\in y}x_{\nu a}$ ($\widetilde{\nu}:Y\lra X$)\\
$yS_\vee x$\>iff\> $\widehat{\nu}(x)\subseteq y$ \> $xS_\wedge y$\>iff\> $\widetilde{\nu}(y)\subseteq x$\\
$zS'_\vee x$\>iff\> $z\upv\widehat{\nu}(x)$ \> $vS'_\wedge y$ \> iff\> $\widetilde{\nu}(y)\upv v$\\
$S_\vee x$\>=\>$\Gamma\widehat{\nu}(x)\in\gphi$ \> $S_\wedge y$ \>=\> $\Gamma\widetilde{\nu}(y)\in\gpsi$\\
$S'_\vee x$\>=\> $\rperp\{\widehat{\nu}(x)\}\in \gpsi$ \> $S'_\wedge y$ \>=\> $\{\widetilde{\nu}(y)\}\rperp\in\gphi$\\
$z\perp_\vee x$ \>iff\> $zS'_\vee x$ \> $y\perp_\wedge v$ \>iff\> $yS'_\wedge v$\\
$\eta_{S_\vee}(U)$\>=\> $\bigcup_{x\in U}S_\vee x$           \> $\eta_{S_\wedge}(V)$\>=\> $\bigcup_{y\in V}S_\wedge y$      \\
$\overline{\eta}_{S_\vee}(A)$\>=\>$\bigvee_{x\in A}S_\vee x\in\gphi$ \>   $\overline{\eta}_{S_\wedge}(B)$\>=\> $\bigvee_{y\in B}S_\wedge y\in\gpsi$    \\[2mm]
$\overline{\eta}_\vee(A)$\>=\> $(\overline{\eta}_{S_\vee}(A))'$  \> $\overline{\eta}_\wedge(A)$\>=\> $\overline{\eta}_{S_\wedge}(A')$\\
$\overline{\eta}_\vee(A)$\>=\>   $A^*$ = $\nu^\pi(A)$ \> $\overline{\eta}_\wedge(A)$\>=\> $\overline{\eta}_{S_\wedge}(A')$ = $\nu^\sigma(A)$
\end{tabbing}
\hrulefill
\end{figure}

We first verify that for  any filter $x$ and ideal $y$, given the definitions in Table~\ref{sigma-pi-canonical},   $\widetilde{\nu}(\widehat{\nu}(x))=x$ and $\widehat{\nu}(\widetilde{\nu}(y))=y$.

The inclusions $x\subseteq\widetilde{\nu}(\widehat{\nu}(x))$ and $y\subseteq \widehat{\nu}(\widetilde{\nu}(y))$ are immediate. Indeed, if $a\in x$, then $\nu a\in\widehat{\nu}(x)=y\in Y$, hence $a=\nu\nu a\in \widetilde{\nu}(\widehat{\nu}(x))$ and similarly for the other inclusion.

To show that $\widetilde{\nu}(\widehat{\nu}(x))\subseteq x$, let $c\in \widetilde{\nu}(\widehat{\nu}(x))$. Then there exist elements $e_1,\ldots,e_n$ in the ideal $v=\widehat{\nu}(x)$ such that $c\in\widetilde{\nu}(v)$ (a filter) because $\nu e_1\wedge\cdots\wedge\nu e_n\leq c$. Using the fact that $\nu$ is an involution, this is equivalent to $\nu c\leq e_1\vee\cdots\vee e_n=e\in v=\widehat{\nu}(x)$, hence $\nu c\in\widehat{\nu}(x)=\bigvee\{y_{\nu a}\midsp a\in c\}$. Let then $a_1,\ldots,a_m\in x$ such that $\nu c\leq\nu a_1\vee\cdots\vee\nu a_m=\nu(a_1\wedge\cdots\wedge a_m)=\nu a$, where $a=a_1\wedge\cdots\wedge a_m\in x$. Using again the involution property of $\nu$, this is equivalent to $a\leq c$ and so $c\in x$.

The argument for the inclusion $\widehat{\nu}(\widetilde{\nu}(y))\subseteq y$ is similar.

The identities established mean that $\widehat{\nu},\widetilde{\nu}$ are injective and inverses of each other. But now we obtain that $xS_\wedge y$ iff $\widetilde{\nu}(y)\subseteq x$ iff $y=\widehat{\nu}(\widetilde{\nu}(y))\subseteq\widehat{\nu}(x)$. Hence, $S_\wedge$ is identical to $R_\wedge$, the derived relation defined for the proof of Lemma~\ref{involution frame lemma} in any frame. By \cite[Lemma~4.3, Lemma~4.6]{duality2}, axioms (I1)--(I3) in Table~\ref{additional axioms table} hold for the canonical relation $S_\wedge$, hence for $R_\wedge$. The identities of axioms (I4) and (I5) of Table~\ref{additional axioms table} were also shown above to hold. We may then appeal to the proof of Lemma~\ref{involution frame lemma} and conclude that $\overline{\eta}_\vee=\overline{\eta}_\wedge$. In other words, $\nu^\sigma=\nu^\pi$ and hence the full complex algebra of Galois sets in the canonical frame is an involutive lattice.
\\[2mm]
{\em Case 4: Orthocomplement.}

If the lattice is an ortholattice, then by the argument for the case of lattices with an involution previously given and by Lemma \ref{ex falso frame lemma} it suffices to verify that the canonical relation ${\perp}=S'_\vee$ is irreflexive. By equation~\eqref{galois dual defn} and definition of $\perp$ as the Galois dual relation $S'_\vee$ of the relation $S_\vee$, $x\perp z$ holds iff there exists a lattice element $e$ such that $e\in z$ and $\nu e\in x$. Reflexivity, $x\perp x$, would then imply that $e\wedge\nu e=0\in x$, contradicting the fact that $x$ is a proper filter.
\\[1mm]
{\em Case 5: Distribution and De Morgan Algebras.}

For the case where the lattice is a De Morgan algebra, i.e. a distributive lattice with an involution, it suffices to prove that $\gpsi$ is distributive. An algebraic proof of this has been given in \cite[Lemma~5.1]{mai-harding}, but we provide here a new proof based on the constructions we have presented.

Note first that both lattice join $\vee$ and meet $\wedge$ are trivially normal lattice operators in the sense of Definition \ref{normal lattice op defn}, but meet is an operator (in the J\'{o}nsson-Tarski sense) only when it distributes over joins. When this is the case, meet also has the distribution type $(1,1;1)$. Its $\sigma$-extension $\wedge_\sigma$, is constructed as outlined in section \ref{normal ops rep section}. Specifically, letting $\wedge=f$, the point operator $\widehat{f}$ on filters is defined by $\widehat{f}(x,z)=\bigvee\{x_{a\wedge b}\midsp a\in x\mbox{ and }b\in z\}$ and the canonical relation $R_\wedge$ is then defined by $xR_\wedge uz$ iff $\forall a,b(a\in u\mbox{ and } b\in z\lra a\wedge b\in x)$.  Note that $R_\wedge$ is the upper bound relation of Theorem \ref{upper bound rel prop}. Considering the image operator $\alpha_R:\powerset(X)\times\powerset(X)\lra\powerset(X)$ defined by $\alpha_R(U,W)=\{x\in X\midsp \exists u\in U\exists z\in W\; xR_\wedge uz\}$,  we obtain that  $\alpha_R(A,C)=A\cap C$, for $A,C\in \gpsi$. All sections of the Galois dual relation of $R_\wedge$ are stable, as they are for any canonical relation, as proven in \cite[Lemma~4.6]{duality2}. It then follows by Theorem \ref{upper bound rel prop} that intersection distributes over arbitrary joins, in other words, $\gpsi$ is a completely distributive lattice.
\\[1mm]
{\em Case 6: Classical complement and Boolean algebras.}

Finally, for the case of Boolean algebras, combine the arguments given for ortholattices  and for De Morgan algebras.
\end{proof}

\begin{coro}[Completeness/Canonicity]\rm
The logic $\mathrm{\Lambda}(\mathbb{V})$, for each of the varieties $\mathbb{V}$ of Figure~\ref{quasi-complements figure}, is complete and canonical.
\end{coro}
\begin{proof}
The proof follows from the representation construction, making use of
Lemma~\ref{restriction to clopens}, Proposition~\ref{canonical extension prop} and Theorem~\ref{fig 1 theorem} for the canonicity claim.
\end{proof}

\section{Spectral Duality}
\label{section: spectral duality}
Let $\mathbf{M,G,INV,DMA,O}$ and $\mathbf{BA}$ be the categories of algebras in the respective varieties $\mathbb{M,G,INV,DMA,O}$ and $\mathbb{BA}$ of Figure \ref{quasi-complements figure} with the usual algebraic homomorphisms.

As in \cite{duality2}, $\mathbf{SRF}_\tau$ designates the category of sorted residuated frames with a relation $R^\sigma$ for each $\sigma\in\tau$. For our present purposes we only consider frames $\mathfrak{F}=(X,\upv,Y,S_\vee)$, with $\sigma(S_\vee)=(\partial;1)$, given that the distribution type of the normal lattice operator $\nu$ under study is $\delta(\nu)=(1;\partial)$. In particular then we let $\mathbf{SRF}_{\nu M}=\mathbf{SRF}_{\{(1;\partial)\}}$ designate the category with objects the sorted residuated frames with a relation $S_\vee$ as above, subject to the axioms of Table \ref{quasiframe axioms table}.
Morphisms are the weak bounded morphisms specified in \cite{duality2} and respecting the frame relation $S$.  The axiomatization for morphisms is presented in Table~\ref{frame axioms nuM}, axioms (M1)--(M5).

The map $\type{F}$ sending a quasi-complemented lattice to its dual frame is the object part of a functor whose action on quasi-complemented lattice homomorphisms $h:\mathbf{L}_2\lra\mathbf{L}_1$ returns the weak bounded morphism $\type{F}(h)=(p,q)$, where $p:\filt(\mathbf{L}_1)\lra\filt(\mathbf{L}_2)$ and $q:\idl(\mathbf{L}_1)\lra\idl(\mathbf{L}_2)$ are each defined as an inverse image map, $p(x)=h^{-1}[x]$ and $q(y)=h^{-1}[y]$. That $\type{F}(h)$ satisfies the morphism axioms (M1)--(M5) of Table~\ref{frame axioms nuM} follows from \cite[Section~4.3, Proposition~4.9]{duality2}, where the general case of an arbitrary normal lattice expansion was treated. Thus $\type{F}:\mathbf{M}\lra\mathbf{SRF}_{\nu M}$ is a properly defined functor. 

\begin{defn}\rm
\label{pi and pi inverse defn}
If $(p,q):(X_2,I_2,Y_2)\ra(X_1,I_1,Y_1)$, where $I_i$ is the complement of $\upv_i$, with $p:X_2\ra X_1$ and $q:Y_2\ra Y_1$,
then we let $\pi=(p,q)$ and we define $\pi^{-1}$
by setting
\[
\pi^{-1}(W)=\left\{\begin{array}{cl}p^{-1}(W)\in\powerset(X_2) &\mbox{ if }W\subseteq X_1\\
q^{-1}(W)\in\powerset(Y_2) & \mbox{ if }W\subseteq Y_1.\end{array}\right.
\]
 Similarly, we let
\[
\pi(w)=\left\{\begin{array}{cl}p(w)\in X_1 &\mbox{ if }w\in X_2\\ q(w)\in Y_1 &\mbox{ if }w\in Y_2.
\end{array}\right.
\]
\end{defn}

For a frame $\mathfrak{F}=(X,\upv,Y,S_\vee)$ we let $\mathrm{L}(\mathfrak{F})=\mathfrak{F}^+$ be its full complex algebra $\gpsi$ of Galois stable subsets of $X$. If $\pi=(p,q):\mathfrak{F}_2\lra\mathfrak{F}_1$ is a frame morphism, let $\mathrm{L}(\pi)=\pi^{-1}:\mathfrak{F}^+_1\lra\mathfrak{F}_2^+$. By \cite[Proposition~3.24, Lemma~3.25]{duality2}, where the general case of normal lattice expansions was treated, we can conclude that $\mathrm{L}(\pi)$ is a homomorphism of quasi-complemented lattices.

The reader is invited to verify that $\type{F}:\mathbf{M}\leftrightarrows\mathbf{SRF}^\mathrm{op}_{\nu M}:\type{L}$ is a dual adjunction. The components  $\eta_\mathbf{L}:\mathbf{L}\ra\mathrm{LF}(\mathbf{L})$ of the unit $\eta:1_\mathbf{IL}\stackrel{\cdot}{\ra}\mathrm{LF}$ are defined by $\eta_\mathbf{L}(a)=X_a=\{x\in\filt(\mathbf{L})\midsp a\in x\}$.
Taking contravariance into account, if $\mathfrak{F}=(X,\upv,Y,S_\vee)$ is a frame,
the component $\varepsilon_\mathfrak{F}:\mathfrak{F}\ra \mathrm{FL}(\mathfrak{F})$ of the counit $\varepsilon:1_{\mathbf{SRF}_{\nu M}}\stackrel{\cdot}{\ra}\mathrm{FL}$ is a pair $\varepsilon_\mathfrak{F}=(\lambda_\mathfrak{F},\rho_\mathfrak{F})$, where  $\lambda_\mathfrak{F}(x)=\{A\in\mathrm{L}(\mathfrak{F})=\gpsi\midsp x\in A\}$ and $\rho_\mathfrak{F}(y)=\{B\in\gphi\midsp y\in B\}$, for $x\in X, y\in Y$.

$\mathbf{SRF}_{\nu M}$ is too large a category for duality purposes and we specify full subcategories for each of the cases of interest. In \cite{duality2}, the notation $\mathbf{SRF}^*_\tau$ was used to designate the intended subcategory and we keep with this notation, while also subscripting appropriately to distinguish between the different categories of interest in this article. For a frame $\mathfrak{F}$ in $\mathbf{SRF}^*_\tau$, we let $\type{L}(\mathfrak{F})$ be the full complex algebra $\mathfrak{F}^+$ of stable sets and $\type{L}^*(\mathfrak{F})$ the complex algebra of clopen elements.

\begin{table}[!t]
\caption{Axioms for $\mathbf{SRF}^*_{\nu\mathrm{M}}$}
\label{frame axioms nuM}
{\bf Object Axioms}
\begin{enumerate}
  \item[(F0)] The complement $I$ of the Galois relation $\upv$ of the frame is quasi-serial, in other words
   $\forall x\in X\exists y\in Y\; xIy$ and $\forall y\in Y\exists x\in X\; xIy$.
  \item[(F1)] The frame is separated.
  \item[(F2)] For each $z\in X$, $S_\vee z$ is a closed element of $\gphi$ and if $z$ is a clopen element (i.e. $\Gamma z={}\rperp\{v\}$ for a (unique, by separation) point $v$ in $Y$), then $S_\vee z$ is a clopen element of $\gphi$.
  \item[(F3)] For each $y\in Y$  the set $yS_\vee$ is decreasing (a down set).
  \item[(F4)] Both sections of the Galois dual relation $S'_\vee$ of $S_\vee$ are Galois sets.
  \item[(F5)] Clopen elements are closed under finite intersections in each of $\gpsi,\gphi$.
  \item[(F6)] The family of closed elements, for each of $\gpsi,\gphi$, is the intersection closure of the respective family of clopens.
  \item[(F7)] Each of $X,Y$ carries a spectral topology generated by the basis of their respective families of clopen elements.
\end{enumerate}
{\bf Weak Bounded Morphism Axioms}\\
  \mbox{\hskip1cm}For a sorted map $\pi=(p,q):(X_2,I_2,Y_2,S_{2\nu})\lra(X_1,I_1,Y_1,S_{1\nu})$, where\\
  \mbox{\hskip1cm}$p:X_2\lra X_1$ and $q:Y_2\lra Y_1$
  \begin{enumerate}
  \item[(M1)] $\forall x'\in X_2\forall y'\in Y_2\;(x'I_2y'\lra p(x')I_1q(y'))$.
  \item[(M2)] $\forall x\in X_1\forall y'\in Y_2(xI_1 q(y')\lra\exists x'\in X_2(x\leq p(x')\wedge x'I_2y')  )$.
  \item[(M3)] $\forall x'\in X_2\forall y\in Y_1(p(x')I_1y\lra\exists y'\in Y_2(y\leq q(y')\wedge x'I_2y'))$.
  \item[(M4)] $\forall z\in X_1\forall v\in Y_2(q(v)S_{1\nu}z\lra \exists x\in X_2(z\leq p(x)\mbox{ and }vS_{2\nu}x))$.
  \item[(M5)] For all points $u$, $\pi^{-1}(\Gamma u)=\Gamma v$, for some (unique, by separation) $v$.
\end{enumerate}
\hrulefill
\end{table}

\begin{thm}\rm
\label{nuM duality}
Let $\mathbf{SRF}^*_{\nu\mathrm{M}}$ be the full subcategory of $\mathbf{SRF}_{\nu M}$ axiomatized by the axioms in Table \ref{frame axioms nuM}.
Then $\type{F}:\mathbf{M}\iso(\mathbf{SRF}^*_{\nu\mathrm{M}})^\mathrm{op}:\type{L}^*$, i.e. the functors ${\tt F,L}^*$ form a categorical duality .
\end{thm}
\begin{proof}
Let $\mathbf{L}=(L,\leq,\wedge,\vee,0,1,\nu)$ be a lattice with a minimal quasi complementation operation.  $\type{F}(\mathbf{L})=(\filt(\mathbf{L}),\upv,\idl(\mathbf{L}),S_\vee)$ is the canonical frame of the lattice constructed in Section~\ref{rep quasi comp section}. Axioms (F0)--(F4) hold for the canonical frame, by Proposition \ref{canonical section stability Snu}. Note that axiom (F2) of Table \ref{frame axioms nuM} is a strengthening of the corresponding axiom in Table \ref{quasiframe axioms table}.

To verify the stronger version of (F2), suppose $\Gamma z={}\rperp\{v\}$ is a clopen element. Clopen elements of $\gpsi$ in the canonical frame are precisely the stable compact-open sets $X_a=\Gamma x_a=x_a{\uparrow}={}\rperp\{y_a\}$. By definition of the point operator $\widehat{\nu}$ and of the canonical relation $S_\vee$ in equation~\eqref{neg point op and rel}, $S_\vee x_a=\Gamma\widehat{\nu}(x_a)$. It is straightforward to see that $\widehat{\nu}(x_a)=y_{\nu a}$, hence $S_\vee x_a=\Gamma y_{\nu a}=Y^{\nu a}$ is a clopen element of $\gphi$, where $y_{\nu a}=(\nu a){\downarrow}$ is the principal ideal generated by the lattice element $\nu a$.

Axiom (F5) holds, since $X_a\cap X_b=X_{a\wedge b}$, while also $Y^a\cap Y^b=Y^{a\vee b}$. For (F6), by join-density of principal filters, any filter $x$ is the join $x=\bigvee_{a\in x}x_a$, hence every closed element $\Gamma x$ of $\gpsi$ is an intersection $\Gamma x=\bigcap_{a\in x}\Gamma x_a=\Gamma\left(\bigvee_{a\in x}x_a\right)$ and similarly for closed elements $\Gamma y\in\gphi$. Finally, axiom (F7) was verified in Proposition \ref{spectral prop} for meet semilattices and the same proof applies to establish that the topology on each of $X=\filt(\mathbf{L})$ and $Y=\idl(\mathbf{L})$ is a spectral topology.

By the above argument, ${\tt F}(\mathbf{L})$ is an object in the category $\mathbf{SRF}^*_{\nu\mathrm{M}}$. By Proposition \ref{compact-open stable prop}, $\{X_a\midsp a\in L\}={\tt KO}\mathcal{G}(\mathfrak{X})$ is the complex algebra of clopen elements $\type{L}^*\type{F}(\mathbf{L})$ and we then verified in Theorem \ref{choice-free lat rep} that the representation map \mbox{$a\mapsto X_a=\{x\in X\midsp a\in x\}$} is a lattice isomorphism $\mathbf{L}\iso{\tt KO}\mathcal{G}(X)$. Since $X^{\nu a}=\Gamma x_{\nu a}=(\Gamma\widehat{\nu}(x_a))'=(\Gamma x_a)^*=(X_a)^*$, the representation map is an isomorphism $\mathbf{L}\iso\type{L}^*\type{F}(\mathbf{L})$ of quasi-complemented lattices with a (minimal) quasi-complementation operation $\nu$.

For morphisms $h:\mathbf{L}_1\lra\mathbf{L}_2$, we let $p=h^{-1}:\filt(\mathbf{L}_2)\lra\filt(\mathbf{L}_1)$ and $q=h^{-1}:\idl(\mathbf{L}_2\lra\idl(\mathbf{L}_1)$ and we set $\type{F}(h)=(p,q)$.

The argument that $\type{F}(h):\type{F}(\mathbf{L}_2)\lra\type{F}(\mathbf{L}_1)$ is a frame morphism satisfying axioms (M1)--(M4) is a special instance of the argument given in the proof of \cite[Proposition~4.9]{duality2}, handling the general case of arbitrary normal lattice expansions. The proofs regarding axioms (M5) and (M6) were given in \cite[Proposition~5.6, Proposition~5.7]{duality2}. This establishes that $\type{F}:\mathbf{M}\lra(\mathbf{SRF}^*_{\nu\mathrm{M}})^\mathrm{op}$ is a contravariant functor satisfying $\mathbf{L}\iso\type{L}^*\type{F}(\mathbf{L})$.

Now let $\mathfrak{F}$ be a sorted residuated frame in the category $\mathbf{SRF}^*_{\nu\mathrm{M}}$. We have let $\type{L}(\mathfrak{F})=\mathfrak{F}^+=(\gpsi,\subseteq,\bigcap,\bigwedge,\emptyset,X,(\;)^*)$ be its full complex algebra of Galois stable sets and $\type{L}^*(\mathfrak{F})=(\type{KO}\mathcal{G}(\mathfrak{X}),\subseteq,\cap,\vee,\emptyset,X,(\;)^*)$ be its subalgebra of clopen elements. That the operation $(\;)^*$ restricts to clopen elements is clear, since $(X_a)^*=X_{\nu a}$. By Corollary \ref{frame conditions for complex algebra coro}, $\type{L}^*(\mathfrak{F})$ is an object in the category $\mathbf{M}$ of lattices with a minimal quasi-complementation operation.

If $\pi=(p,q):\mathfrak{F}_2\lra\mathfrak{F}_1$ is a frame morphism in $\mathbf{SRF}^*_{\nu\mathrm{M}}$, then by \cite[Corollary~3.21]{duality2}, $\type{L}^*(\pi)=\pi^{-1}:\mathcal{G}(X_1)\lra\mathcal{G}(X_2)$  is a complete lattice homomorphism of the complete lattices of stable sets of the frames. Given axiom (M4), it was established in \cite[Proposition~3.24, Lemma~3.25]{duality2} that $\pi^{-1}:\mathfrak{F}_1^+\lra\mathfrak{F}_2^+$ is in fact a homorphism of the full complex algebras of the frames. By axiom (M5), $\pi^{-1}$ preserves closed elements, hence  by \cite[Lemma~3.23]{duality2} it also preserves clopen elements (from which continuity of $\pi$ follows, since clopen stable elements are precisely the basic open sets in the topology).

The above argument has established that $\type{L}^*$ is a contravariant functor from the category $\mathbf{SRF}^*_{\nu\mathrm{M}}$ to the category $\mathbf{M}$ of lattices with a minimal quasi-complementation operation. We have already also established in Lemma~\ref{restriction to clopens} that for any object $\mathbf{L}$ in $\mathbf{M}$ we have an isomorphism $\mathbf{L}\iso\type{L}^*\type{F}(\mathbf{L})$ and it remains to argue that for any sorted residuated frame $\mathfrak{F}$ in the category $\mathbf{SRF}^*_{\nu\mathrm{M}}$ we also have that $\mathfrak{F}\iso\type{F}\type{L}^*(\mathfrak{F})$. 

For the second isomorphism, $\mathfrak{F}\iso\type{F}\type{L}^*(\mathfrak{F})$, we may base the argument either on \cite{sdl}, or on \cite{discr}. In \cite{sdl}, we argued that any sorted frame $\mathfrak{F}$ (L-frame in \cite{sdl}) subject to the axioms F1, F5--F7 is the frame dual to a lattice. More specifically, if $\mathcal{S}=(C_s)_{s\in L}$ is the family of compact-open Galois stable subsets of $X$, indexed in some set $L$, and since $\mathcal{S}$ is in fact a normal lattice expansion of similarity type $\tau$ (by the frame axioms and by the fact that $(\;)^*$ restricts to clopen elements, which are precisely the compact-open Galois stable subsets), then the indexing set $L$ inherits the structure of the family $\mathcal{S}$. Hence, it is a normal lattice expansion $\mathcal{L}$ of type $\tau$ isomorphic to $\mathcal{S}$ and $\mathfrak{F}$ is, up to isomorphism, its dual frame. In other words, $\type{L}^*(\mathfrak{F})=\mathcal{S}\iso\mathcal{L}$ and $\mathfrak{F}\iso \type{F}(\mathcal{L})$.   Therefore,  we have $\type{F}\type{L}^*(\mathfrak{F})\iso  \type{F}\type{L}^*\type{F}(\mathcal{L})\iso\type{F}(\mathcal{L})\iso\mathfrak{F}$.
\end{proof}

\begin{defn}\label{frame cats defn}\rm
  Categories of frames $\mathfrak{F}=(X,\upv,Y,S_\vee)$, where we set $\perp=S'_\vee$, corresponding to lattices with a quasi complementation operation are axiomatized by the axioms of Table \ref{frame axioms nuM} as well as one or more of the additional axioms below, as specified for each category.
  \begin{tabbing}
  (G) \hskip7mm\= $\perp$ is symmetric.\\
  (INV) \> Axioms (I1)--(I5) of Table~\ref{additional axioms table}.\\
  (O) \> $\perp$ is irreflexive.\\
  (D) \> All sections of the Galois dual relation $R'_\leq$ of the upper bound relation\\
  \>$R_\leq$ of Theorem \ref{upper bound rel prop} are stable.
  \end{tabbing}
  \begin{tabbing}
    $\mathbf{SRF}^*_{\nu\mathrm{M}}$\hskip7mm\= Table \ref{frame axioms nuM} axioms.\\
    $\mathbf{SRF}^*_{\nu\mathrm{G}}$\> Table \ref{frame axioms nuM} axioms + Axiom (G).\\
    $\mathbf{SRF}^*_{\nu\mathrm{INV} }$\> Table \ref{frame axioms nuM} axioms + Axiom (G) + Axioms (INV).\\
    $\mathbf{SRF}^*_{\nu\mathrm{O} }$\> Table \ref{frame axioms nuM} axioms + Axiom (G) + Axioms (INV) + Axiom (O).\\
    $\mathbf{SRF}^*_{\nu\mathrm{DMA}}$\> Table \ref{frame axioms nuM} axioms + Axiom (G) + Axioms (INV) + Axiom (D).\\
    $\mathbf{SRF}^*_{\nu\mathrm{BA}}$\> Table \ref{frame axioms nuM} axioms + Axiom (G) + Axioms (INV) +
       Axiom (O) +\\ \> + Axiom (D).
  \end{tabbing}
\end{defn}

\begin{thm}\rm\label{dualities thm}
The spectral duality of Theorem \ref{nuM duality} specializes to dualities for each of the frame categories of Definition \ref{frame cats defn} and their respective categories of bounded lattices with a (quasi) complementation operation.
\end{thm}
\begin{proof}
That the double dual of a lattice in one of the varieties of Figure \ref{quasi-complements figure} is in the variety in question was verified in Theorem \ref{fig 1 theorem}. That the (full) complex algebra of a frame in one of the frame categories which satisfies one or more axioms from the list in Definition \ref{frame cats defn} is an algebra in the respective variety corresponding to the frame category was verified in Corollary \ref{frame conditions for complex algebra coro}. The rest of the duality argument for each of the cases is the same as in Theorem \ref{nuM duality}.
\end{proof}

\section{Concluding Remarks}
\label{final rems section}
In this article, we have provided alternative constructions for the choice-free representation and duality for Boolean algebras and Ortholattices, first given in \cite{choice-free-BA,choice-free-Ortho}. A Stone duality result for De Morgan algebras, using choice, was published by Bimb\'{o} in \cite{kata-ortho} and we have given here a choice-free version of the duality. 

Our background motivation has been the J\'{o}nsson-Tarski \cite{jt1,jt2} approach, constructing set-operators from relations to represent operators on Boolean algebras, a project that has been extended with Dunn's research on generalized Galois logics (gaggles). In \cite{duality2}, we presented a generalization of this project of relational representation to cases where distribution may not be assumed. The framework of \cite{duality2} was applied in this article to the case of bounded lattices with a quasi-complementation operator, recasting the duality of \cite{duality2} in a choice-free manner. By treating both De Morgan (and Boolean) algebras, as well as Ortholattices, it has been shown that the presence or not of distribution does not create any significant obstacle, as distribution in the complete lattice of stable sets of a sorted frame has been shown to be first-order definable. For the distributive case, the resulting semantics from the approach presented appears to have strong affinity to Holliday's possibility semantics \cite{holliday-posibility-frames}.

The approach presented can be extended to any normal lattice expansion, including the case of modal lattices studied in \cite{choice-free-dmitrieva-bezanishvili}, based on the framework developed in \cite{duality2}.

A variant choice-free semantic approach, based on Urquhart's representation, would be to consider in the lattice representation all disjoint filter-ideal pairs $(F,I)$, rather than the maximal ones only. Thereby the appeal to choice in the proof of existence of maximally disjoint pairs is no longer necessary.  This project has been in fact carried out by Allwein and this author in \cite{iulg-bounded} and it has been considered again and used in the recent semantic approach of Holliday \cite{holliday-fundamental}. In its application to logics with weak negation a further restriction is added in \cite{holliday-fundamental} requiring that any point $(F,I)$ of the canonical frame satisfies the condition $\{\neg a\midsp a\in F\}\subseteq I$ and a representation result \cite[Theorem~4.30]{holliday-fundamental} is reported. Duality is not discussed in \cite{holliday-fundamental}, but we believe that the lattice duality of \cite{iulg-bounded} can be of use in Holliday's approach.

Related to the present article is also Almeida \cite{Almeida09}, where the author
 examines representation issues for lattices with various negation operators, largely overlapping with the cases studied in this article, but working in the framework of RS-frames (generalized Kripke frames \cite{mai-gen}).
As in the present article, results are presented on canonicity of negation axioms \cite[Proposition~3.6 -- Proposition~3.12]{Almeida09} and first-order correspondents are calculated. The lattice representation assumed in the RS-frames approach is Hartung's \cite{hartung} and the proof that the set of stable sets is a canonical extension of the represented lattice depends on proving that the set $J^\infty(X)$ of completely join-irreduciple stable sets join-generates all stable sets. Since the set of closed elements in the frame is join-dense in the lattice of stable sets, the proof reduces to proving that $J^\infty(X)$  join generates all closed elements in the lattice of stable sets. This proof is given in \cite[Theorem~2.8]{dunn-gehrke} where an explicit use of Zorn's lemma is made.  Correspondence results in \cite{Almeida09} are established using the discrete duality between perfect lattices and RS-frames detailed in \cite[Section~4]{dunn-gehrke}, where recall that a lattice $\mathbf{L}$ is perfect if it is a complete lattice, $J^\infty(\mathbf{L})$ is join-dense in $\mathbf{L}$ and $M^\infty(\mathbf{L})$ (the set of completely meet irreducible elements of $\mathbf{L}$) is meet-dense in $\mathbf{L}$. All correspondence results proven in \cite[Theorem~4.4 -- Theorem~4.9]{Almeida09} are based on the above discrete duality result of \cite{dunn-gehrke}.

As we pointed out in \cite{kata2z}, the RS-frames approach and the approach taken by this author can be used equally well in studying specific problems related to non-distributive systems and the choice depends on the objectives at hand. In the case of the present article, working without any appeal to the axiom of choice, or any of its equivalents, has been the main objective, hence working with RS-frames  was not a suitable approach.

\end{document}